%% file: main.tex
\documentclass[11pt]{article}

\usepackage[margin=1in]{geometry}
\usepackage{amsmath,amssymb,amsthm,mathtools,bm}
\usepackage{aliascnt}
\usepackage{booktabs,tabularx}
\usepackage{graphicx}
\usepackage{subcaption}
\usepackage{flafter}
\usepackage{placeins}
\usepackage{enumitem}
\usepackage{microtype}
\usepackage{xcolor}
\usepackage[breakable]{tcolorbox}
\usepackage{tikz}
\usepackage[authoryear,round,sort]{natbib}
\usepackage[colorlinks=true,linkcolor=blue!55!black,
            citecolor=blue!55!black,urlcolor=blue!55!black]{hyperref}
\makeatletter
\def\UrlBreaks{\do\.\do\@\do\\\do\!\do\_\do\|\do\;\do\>\do\]%
  \do\)\do\,\do\?\do\&\do\'\do+\do\=\do\#}
\def\UrlBigBreaks{\do@url@hyp}
\g@addto@macro\UrlOrds{\do\/\do\:}
\makeatother
\usepackage[nameinlink,capitalise,noabbrev]{cleveref}

\allowdisplaybreaks \setlist[itemize]{leftmargin=1.7em,itemsep=0.25em,topsep=0.35em} \setlist[enumerate]{leftmargin=1.9em,itemsep=0.25em,topsep=0.35em} \newcolumntype{Y}{>{\raggedright\arraybackslash}X}

\newtheorem{theorem}{Theorem}[section] \newaliascnt{corollary}{theorem} \newtheorem{corollary}[corollary]{Corollary} \aliascntresetthe{corollary} \newaliascnt{proposition}{theorem} \newtheorem{proposition}[proposition]{Proposition} \aliascntresetthe{proposition} \newaliascnt{lemma}{theorem} \newtheorem{lemma}[lemma]{Lemma} \aliascntresetthe{lemma}
\newaliascnt{definition}{theorem} \theoremstyle{definition} \newtheorem{definition}[definition]{Definition} \aliascntresetthe{definition} \theoremstyle{plain}
\newaliascnt{remark}{theorem}  \aliascntresetthe{remark} \newaliascnt{assumption}{theorem}  \aliascntresetthe{assumption}
\newaliascnt{example}{theorem} \theoremstyle{definition} \newtheorem{example}[example]{Example} \aliascntresetthe{example} \theoremstyle{plain}

\newcommand{\R}{\mathbb{R}}
\newcommand{\E}{\mathbb{E}}
\newcommand{\Pp}{\mathbb{P}}
\newcommand{\one}{\bm{1}}
\newcommand{\ind}{\mathbf{1}}
\newcommand{\e}{\mathrm{e}}
\newcommand{\wh}{\widehat}
\newcommand{\wt}{\widetilde}
\newcommand{\cA}{\mathcal{A}}
\newcommand{\cD}{\mathcal{D}}

\newcommand{\cH}{\mathcal{H}}
\newcommand{\cL}{\mathcal{L}}

\newcommand{\cV}{\mathcal{V}}
\newcommand{\norm}[1]{\left\lVert #1\right\rVert}
\newcommand{\abs}[1]{\left\lvert #1\right\rvert}

\newcommand{\supp}{\operatorname{supp}}
\newcommand{\Top}{\operatorname{Top}}

\newcommand{\sig}{\sigma}
\newcommand{\MLE}{\mathrm{MLE}}
\newcommand{\RC}{\mathrm{RC}}

\newcommand{\runinhead}[1]{%
  \par\addvspace{3.25ex plus 1ex minus .2ex}%
  \noindent\textbf{#1.}\quad\ignorespaces}

\newif\ifanonymous \anonymousfalse

\title{Graph-monotone entrywise guarantees for MLE and Rank Centrality on general comparison graphs} \ifanonymous
  \author{Anonymous authors}
\else
  \author{Yuepeng Yang$^{1}$ \qquad Cong Ma$^{2}$\\[3pt]
    \small $^{1}$Department of Statistics and Data Science, the Wharton School,\\[-1pt] \small University of Pennsylvania\\[1pt] \small $^{2}$Department of Statistics, University of Chicago\\[2pt]} \fi
\date{October 6, 2026}
\hypersetup{
  pdftitle={Graph-monotone entrywise guarantees for MLE and Rank Centrality on general comparison graphs},
}
\ifanonymous
  \hypersetup{pdfauthor={Anonymous authors}}
\else
  \hypersetup{pdfauthor={Yuepeng Yang and Cong Ma}}
\fi

\begin{document}
\maketitle

\begingroup
  \renewcommand{\thefootnote}{}
  \def\UrlBreaks{\do\/\do\-\do\.}
  \footnotetext[0]{Experiment code and Lean formalization:
    \url{https://github.com/YuepengYang/pairwise-ranking-general-graphs}.}
\endgroup

\begin{abstract}
  Pairwise comparisons are widely used to infer latent scores and identify top-ranked items.
  Although sharp statistical guarantees are available under uniform sampling, real data often induce irregular comparison graphs with heterogeneous observation counts across pairs.
  In this paper, we study an arbitrary fixed comparison graph under the Bradley--Terry--Luce model with minimal assumptions.
  We prove that both the standard maximum likelihood estimator and Rank Centrality achieve a high-probability entrywise error rate of order $1/\sqrt{\lambda_{\cD}}$ up to logarithmic and dynamic-range factors, where $\lambda_{\cD}$ is the algebraic connectivity of the count-weighted observation graph.
  This guarantee is graph-monotone since $\lambda_{\mathcal D}$ cannot decrease when additional comparisons are added. 
  Under heterogeneous sampling, where comparison pairs are sampled independently with unequal probabilities, our guarantee improves existing error bounds or requires weaker assumptions.
  We further extend our analysis to show that both estimators are robust against outcome-adaptive augmentation, where an adversary can choose additional comparison pairs after observing the initial outcomes. 
\end{abstract}

\input{sections/01_introduction}
\input{sections/02_preliminaries}
\input{sections/03_entrywise_guarantees}
\input{sections/04_outcome_adaptive_augmentation}

\input{sections/05_experiments}
\input{sections/06_discussion}

\subsection*{Acknowledgments}
C.M.~was partially supported by the National Science Foundation via the CAREER
Award \mbox{DMS-2443867}.

\subsection*{AI disclosure}
AI tools were used to assist in the preparation of this manuscript.
The authors take full responsibility for the accuracy of its content, including all mathematical arguments, numerical results, and references.

\bibliographystyle{plainnat}
\bibliography{references}

\appendix
\input{sections/a_shared_tools}

\input{sections/b_mle_proof}
\input{sections/c_rank_centrality_proof}
\input{sections/d_outcome_adaptive_proofs}
\input{sections/e_experimental_implementation}

\end{document}

%% file: sections/01_introduction.tex
\section{Introduction}
\label{sec:introduction}

Pairwise comparisons are widely used to estimate latent scores and identify highly ranked items in applications such as sports tournaments \citep{Zermelo1929,Cattelan2012}, choice experiments \citep{Grossmann2019PartialProfile}, and comparative evaluation of large language models \citep{ChiangEtAl2024Arena}.

A classical and popular way to model pairwise comparisons is the Bradley--Terry--Luce (BTL) model \citep{BradleyTerry1952,Luce1959}.
Given $n$ items, the model assumes that each item $i$ has a latent score $\theta_i^\star$ and that item $i$ defeats item $j$ with probability
\begin{equation}
  \Pp\{i\text{ defeats }j\}
  =\frac{\exp(\theta_i^\star)}
  {\exp(\theta_i^\star)+\exp(\theta_j^\star)}.
  \label{eq:intro-btl}
\end{equation}
Given a number of observed pairwise comparisons, a central statistical task is to aggregate local pairwise comparison outcomes into accurate global estimates of the latent scores $\bm\theta^\star=(\theta_1^\star,\ldots,\theta_n^\star)$ and identify the top-$K$ items.

Most sharp entrywise guarantees for these tasks assume that comparison pairs are sampled according to an Erd\H{o}s--R\'enyi graph \citep[e.g.,][]{ChenSuh2015,JangKimSuhOh2016,ChenFanMaWang2019,ChenGaoZhang2022}, so that each edge is sampled independently with the same probability.
In reality, the edge sampling process can be less regular and possibly dependent. 
An entrywise guarantee for a general comparison design is therefore desirable.
Most existing results for general comparison graphs, however, either require strong structural assumptions or provide weaker guarantees.
To address this gap, we consider the following research question:

\begin{quote}
  Can we establish a strong entrywise guarantee for general comparison designs under minimal conditions?
\end{quote}
In particular, we want to understand what properties of the comparison design control the entrywise estimation error.

\subsection{Our contributions}
\label{sec:intro-contributions}
We now present an overview of our main results.
Consider a general connected comparison graph $G=(V,E)$ with $n$ vertices, where each edge $\{i,j\}\in E$ represents a pair of items that is independently compared $L_{ij}$ times.
We analyze two classical algorithms, the maximum likelihood estimator (MLE) and Rank Centrality, for entrywise score estimation and exact top-$K$ recovery.
For clarity, we suppress dynamic-range factors here and leave them to the formal statements.
\runinhead{Graph-monotone entrywise guarantees on general comparison designs}
For every fixed comparison design $\cD=(G,\{L_{ij}\})$, we establish a common entrywise error guarantee for the MLE and Rank Centrality.
Define the count-weighted observation Laplacian and its algebraic connectivity by
\begin{equation}
  \bm L_{\cD}
  :=\sum_{\{i,j\}\in E}L_{ij}
  (\bm e_i-\bm e_j)(\bm e_i-\bm e_j)^\top,
  \qquad
  \lambda_{\cD}:=\lambda_2(\bm L_{\cD}).
  \label{eq:intro-observation-laplacian}
\end{equation}
Here $\lambda_2$ is the second-smallest eigenvalue.
We show that whenever $\lambda_{\cD}\gtrsim\log n$, both estimators are well defined with high probability and satisfy
\begin{equation}
  \norm{\wh{\bm\theta}^{a}-\bm\theta^\star}_\infty
  \lesssim
  \frac{\log^{3/2}n}{\sqrt{\lambda_{\cD}}},
  \qquad a\in\{\MLE,\RC\}.
  \label{eq:intro-rate}
\end{equation}
In particular, they exactly recover the top-$K$ set when $\Delta_K^2\gtrsim\log^3n/\lambda_{\cD}$, where $\Delta_K$ is the gap between the $K$th and $(K+1)$th largest scores.
Compared with existing results for general comparison graphs, our results provide a strong guarantee while requiring only minimal dynamic-range assumptions and the condition $\lambda_\cD\gtrsim \log n$.
We show through the example of heterogeneous sampling, where each edge is sampled with a probability in $[p, q]$, that our guarantee either gives sharper error bounds or requires weaker assumptions than existing results.
Moreover, the guarantees are graph-monotone, in the sense that adding comparisons can only increase $\lambda_{\cD}$ and therefore improve the bound.

\runinhead{Robustness to outcome-adaptive augmentation}
We extend our general-design guarantees to outcome-adaptive augmentation.
We consider the scenario where an adversary observes the initial comparison outcomes and chooses additional comparisons, while the estimator only sees the final augmented data.
We show that both estimators obtain the entrywise guarantee \eqref{eq:intro-rate} with the algebraic connectivity of the final design.
This shows that both MLE and Rank Centrality are robust to outcome-adaptive adversarial augmentation.

In particular, this applies to the semirandom sampling model, where the initial comparison design is an Erd\H{o}s--R\'enyi graph with $L$ comparisons per edge.
Our result resolves the open problem in \citet{YangChenOrecchiaMa2024}, showing that reweighting is not necessary to address this problem.

\subsection{Related work}
\label{sec:intro-related}

\paragraph*{BTL model and entrywise guarantees.}

The BTL model provides a classical probabilistic formulation of ranking from binary comparisons \citep{BradleyTerry1952,Luce1959}.
Two standard estimators for this model are the maximum likelihood estimator and Rank Centrality \citep{Ford1957,NegahbanOhShah2017}.
Under homogeneous Erd\H{o}s--R\'enyi sampling and bounded dynamic range, Rank Centrality and MLE attain sharp entrywise error bounds and minimax-optimal sample complexity for exact top-$K$ recovery \citep{ChenSuh2015,JangKimSuhOh2016,ChenFanMaWang2019,HanYeTanChen2020,GaoShenZhang2023}.

\runinhead{Ranking on general comparison graphs}
Going beyond uniform sampling, existing estimation guarantees impose graph regularity or sample-size conditions \citep{YanYangXu2012,Chen2023General,HanXuChen2023,FanLouWangYu2026,LiXu2026}, or involve degree parameters in addition to algebraic connectivity \citep{LiShrotriyaRinaldo2022}.
Our results give graph-monotone entrywise guarantees for both MLE and Rank Centrality with graph dependence only through the count-weighted algebraic connectivity.

\runinhead{Heterogeneous and semirandom sampling}
Recent work studies ranking from pairwise comparisons in the middle ground between uniform sampling and fully general comparison designs.
\citet{YangChenOrecchiaMa2024} obtain top-$K$ guarantees using an SDP-reweighted MLE for the semirandom setting, while \citet{LeeMakurSingh2026} analyze Rank Centrality under heterogeneous sampling, requiring additional regularity for the unweighted estimator and obtaining a suboptimal bound in sparse regimes for the weighted estimator.
In contrast, our result establishes robustness to outcome-adaptive augmentation of general initial designs for the standard MLE and Rank Centrality without reweighting.

Adding comparisons based on existing observations also arises in active ranking and synthetic data augmentation.
Active ranking algorithms adaptively select further comparisons to improve sample efficiency \citep{MohajerSuhElmahdy2017,GoldbergerRudi2026}, while \citet{JinYu2026} analyze MLE under iterative augmentation with synthetic comparisons generated from fitted BTL models.
These works use beneficial data augmentation, while our semirandom guarantee handles adversarial augmentation.

\subsection{Organization}
\Cref{sec:model} introduces the model and the two estimators.
\Cref{sec:results} gives the general-graph entrywise error bounds, as well as an application and the analysis.
\Cref{sec:random-design-models} establishes robustness to outcome-adaptive augmentation, with semirandom and heterogeneous sampling as special cases.
\Cref{sec:experiments} reports the real-data and numerical studies.
Finally, \Cref{sec:discussion} concludes with a discussion of future directions.

\subsection{Notation}
For an integer $n\ge1$, let $[n]=\{1,\ldots,n\}$.
For a finite set $S$, we write $|S|$ for its cardinality and $S^c$ for its complement.
For a scalar $x$, let $(x)_+:=\max\{x,0\}$.
We write $\Pp$ and $\E$ for probability and expectation, respectively, and $\ind\{\cA\}$ for the indicator of an event $\cA$.
We use bold lowercase letters for vectors and bold uppercase letters for matrices.
Here, $\bm e_i$ is the $i$th standard basis vector, $\one_n$ is the all-ones vector, and $\bm I$ is the identity matrix.
For a vector $\bm x$, $\norm{\bm x}_p$ denotes its $\ell_p$ norm and $\supp(\bm x):=\{i:x_i\ne0\}$, while $\norm{\bm A}$ denotes the operator norm of a matrix.
Scalar functions applied to vectors are understood coordinatewise, and $\log$ denotes the natural logarithm.
For symmetric matrices, $\bm A\succeq\bm B$ means that $\bm A-\bm B$ is positive semidefinite. 
Unless explicitly fixed, $c,C>0$ denote universal constants whose values may change from line to line.
For positive quantities $a$ and $b$, we write $a\lesssim b$ when $a\le Cb$, $a\gtrsim b$ when $b\lesssim a$, and $a\asymp b$ when both relations hold.
For positive sequences, we write $a\ll b$ when $a=o(b)$, meaning that $a/b\to0$ along the asymptotic regime under consideration, and $a\gg b$ when $b\ll a$.
We use $\{\cdot,\cdot\}$ to denote an unoriented edge and $(\cdot,\cdot)$ to denote an oriented edge.
Finally, $\widetilde O(\cdot)$ suppresses polylogarithmic factors in the relevant problem dimensions.

%% file: sections/02_preliminaries.tex
\section{Preliminaries}
\label{sec:model}

\subsection{BTL model and exact top-\texorpdfstring{$K$}{K} identification}

Assuming $n\ge3$, let $G=([n],E)$ be an arbitrary fixed connected undirected graph.
For each edge $e=\{i,j\}$, fix an arbitrary orientation $e=(i,j)$ and write $\bm b_e=\bm e_i-\bm e_j$. 
Let $L_e=L_{ij}=L_{ji}\ge1$ denote the number of comparisons on edge $e$. 
We denote the comparison design by $\cD=(G,\{L_e\}_{e\in E})$ and its weighted Laplacian and algebraic connectivity by
\begin{equation}
  \bm L_{\cD}
  :=\sum_{\{i,j\}\in E}
  L_{ij}\cdot(\bm e_i-\bm e_j)(\bm e_i-\bm e_j)^\top, \qquad \lambda_{\cD}:=\lambda_2(\bm L_{\cD}).
  \label{eq:observation-laplacian}
\end{equation}
Since $G$ is connected, $\lambda_{\cD}>0$.

For identifiability, let the latent score vector $\bm\theta^\star$ lie in $\cH:=\{\bm\theta\in\R^n:\one_n^\top\bm\theta=0\}$.
We write $\alpha_i^\star:=\exp(\theta_i^\star)$ and define the dynamic range as
\begin{equation}
  \kappa:=\frac{\max_i\alpha_i^\star}{\min_i\alpha_i^\star}
  =\exp\!\left(\max_i\theta_i^\star-\min_i\theta_i^\star\right).
  \label{eq:dynamic-range}
\end{equation}
Here, $\kappa$ is the dynamic range of the latent scores.

On each oriented edge $e=(i,j)$, we observe
\begin{equation}
  Y_e^{(1)},\ldots,Y_e^{(L_e)} \stackrel{\mathrm{ind}}{\sim} \operatorname{Bernoulli}\!
  \left(
  \sig(\bm b_e^\top\bm\theta^\star)
  \right),
  \qquad
  \overline Y_e \coloneqq L_e^{-1}\sum_{\ell=1}^{L_e}Y_e^{(\ell)}, 
  \label{eq:btl}
\end{equation}
independently across edges and repetitions, where $\sigma(x) := e^x/(1+e^x)$ is the sigmoid function.

Fix $1\le K\le n-1$ and let $\theta_{(1)}^\star\ge\cdots\ge\theta_{(n)}^\star$ denote the ordered coordinates of $\bm\theta^\star$.
Assume $\theta_{(K)}^\star>\theta_{(K+1)}^\star$ and define
\begin{equation}
  S_K^\star:=\{i\in[n]:\theta_i^\star\ge\theta_{(K)}^\star\},
  \qquad
  \Delta_K:=\theta_{(K)}^\star-\theta_{(K+1)}^\star>0.
  \label{eq:topk-gap-definition}
\end{equation}
For any score estimate $\wh{\bm\theta}$, let $\Top_K(\wh{\bm\theta})$ be the indices of its $K$ largest coordinates.
Our goal is to estimate the latent scores $\bm\theta^\star$ and identify the top-$K$ set from the observed comparison outcomes.

\subsection{MLE and Rank Centrality}
We introduce two standard estimators for the latent scores: the maximum likelihood estimator and Rank Centrality.

\paragraph*{Maximum likelihood estimator.}
The negative log-likelihood is
\begin{equation}
  \cL(\bm\theta)
  :=\sum_{e\in E}
  L_e\left\{ \log\!
  \left(1+\e^{\bm b_e^\top\bm\theta}\right)
  -\overline Y_e\,\bm b_e^\top\bm\theta
  \right\}.
  \label{eq:mle-loss}
\end{equation}
The MLE is
\begin{equation}
  \wh{\bm\theta}^{\MLE}
  \in\arg\min_{\bm\theta\in\cH}\cL(\bm\theta).
  \label{eq:mle-estimator}
\end{equation}

\paragraph*{Rank Centrality.}
For an observed edge with fixed orientation $e=(i,j)$, define $\wh q_{ij}\coloneqq 1-\overline Y_e$ and $\wh q_{ji}:=1-\wh q_{ij}=\overline Y_e$.
Choose any $d\ge \max_{i\in[n]}\sum_{j:\{i,j\}\in E}L_{ij}$ and define the row-stochastic matrix $\wh{\bm P}(d)=(\wh P_{ij}(d))_{i,j\in[n]}$ by
\begin{equation}
  \wh P_{ij}(d)=\frac{L_{ij}\wh q_{ij}}{d}
  \quad(i\ne j,\ \{i,j\}\in E),
  \qquad
  \wh P_{ii}(d)=1-\sum_{j\ne i}\wh P_{ij}(d),
  \label{eq:rank-centrality-matrix}
\end{equation}
with all other entries set to zero.
This row-stochastic matrix defines the empirical Rank Centrality Markov chain on the items.
The empirical transition matrix is a plug-in estimator of its population counterpart.
Define
\[
  q_{ij}^\star:=\E[\wh q_{ij}]
  =\frac{\alpha_j^\star}{\alpha_i^\star+\alpha_j^\star},
  \qquad
  \pi_i^\star:=\frac{\alpha_i^\star}{\sum_{k=1}^n\alpha_k^\star},
\]
and let $\bm P^\star(d)$ be the population transition matrix obtained by replacing $\wh q_{ij}$ with $q_{ij}^\star$ in \eqref{eq:rank-centrality-matrix}.
For every observed edge $\{i,j\}$, $ \pi_i^\star P_{ij}^\star(d) =\pi_j^\star P_{ji}^\star(d)$, so detailed balance shows that $\bm\pi^\star$ is stationary for $\bm P^\star(d)$.
Moreover, the centering constraint on $\bm\theta^\star$ gives
\[
  \log\pi_i^\star-\frac1n\sum_{j=1}^n\log\pi_j^\star
  =\theta_i^\star.
\]
Rank Centrality applies the same idea to the empirical chain $\wh{\bm P}(d)$.
When the empirical chain is irreducible, it has a unique stationary distribution $\wh{\bm\pi}$ with strictly positive entries, satisfying $\wh{\bm\pi}^\top\wh{\bm P}(d)=\wh{\bm\pi}^\top$ and $\one_n^\top\wh{\bm\pi}=1$.
It is easy to see that $\wh \pi$ is the same for any $d\ge \max_{i\in[n]}\sum_{j:\{i,j\}\in E}L_{ij}$.
On this event, the Rank Centrality estimator is uniquely defined by
\begin{equation}
  \wh\theta_i^{\RC}
  :=\log\wh\pi_i-\frac1n\sum_{j=1}^n\log\wh\pi_j.
  \label{eq:rc-estimator}
\end{equation}

%% file: sections/03_entrywise_guarantees.tex
\section{Entrywise guarantees on general designs}
\label{sec:results}
In this section, we present the main entrywise guarantees, an overview of the proof strategy, and an application.
To simplify the notation, we use $C>0$ to denote a universal constant whose value may change from line to line throughout the paper.

\subsection{Entrywise guarantees and \texorpdfstring{top-$K$}{top-K} recovery}

Recall that $\bm{L}_{\cD}$ is the Laplacian defined in \eqref{eq:intro-observation-laplacian} and $\lambda_{\cD} = \lambda_2(\bm{L}_{\cD})$ is the algebraic connectivity of $\cD$.
\begin{theorem}[Entrywise guarantees for MLE and Rank Centrality]
  \label{thm:general-entrywise}
  Fix $n\ge3$, a connected comparison design $\cD=(G,\{L_e\}_{e\in E})$ with $G=([n],E)$ and positive integer counts $L_e$, and $\bm\theta^\star\in\cH$ with dynamic range $\kappa$.
  Let $C$ and $C_{\mathrm{inf}}$ be sufficiently large absolute constants.
  Suppose that   
  \begin{equation}
    \lambda_{\cD}\ge C\kappa^2\log n.
  \label{eq:general-baseline}
  \end{equation}
  Then, with probability at least $1-Cn^{-10}$, both the MLE and Rank Centrality estimators exist uniquely and satisfy
  \begin{equation}
    \max_{a\in\{\MLE,\RC\}}
    \norm{\wh{\bm\theta}^{a}-\bm\theta^\star}_\infty
    \le
    C_{\mathrm{inf}}\kappa\frac{\log^{3/2}n}{\sqrt{\lambda_{\cD}}}.
    \label{eq:general-entrywise}
  \end{equation}
\end{theorem}

These $\ell_\infty$ error guarantees naturally imply exact top-$K$ recovery whenever the gap $\Delta_K$ in \eqref{eq:topk-gap-definition} is sufficiently large relative to the entrywise error.

\begin{corollary}[Exact top-$K$ recovery]
  \label{cor:general-topk}
  Under \eqref{eq:general-baseline}, both the MLE and Rank Centrality exactly recover the top-$K$ set with probability at least $1-Cn^{-10}$ whenever
  \begin{equation}
    \lambda_{\cD}
    >
    C \kappa^2\log n\cdot \max\left\{1, \frac{\log^2 n}{\Delta_K^2}\right\}.
    \label{eq:general-topk-samples}
  \end{equation}
\end{corollary}

\begin{proof}
  The condition \eqref{eq:general-topk-samples} implies \eqref{eq:general-baseline}, so \Cref{thm:general-entrywise} applies.
  For either estimator $a\in\{\MLE,\RC\}$ and any $i\in S_K^\star$ and $j\notin S_K^\star$, $\wh\theta_i^{a}-\wh\theta_j^{a} \ge \theta_i^\star-\theta_j^\star -2\|\wh{\bm\theta}^{a}-\bm\theta^\star\|_\infty$.
  By \eqref{eq:topk-gap-definition}, $\theta_i^\star-\theta_j^\star \ge\Delta_K$.
  Combining this with \eqref{eq:general-entrywise} and \eqref{eq:general-topk-samples} proves the claim for both estimators.
\end{proof}

\paragraph*{Graph-monotone general-graph guarantees via $\lambda_\cD$.}
Instead of the usual random sampling setting \citep{ChenSuh2015,JangKimSuhOh2016,ChenFanMaWang2019,ChenGaoZhang2022}, where each edge is sampled independently with equal probability $p$, \Cref{thm:general-entrywise} provides entrywise error guarantees for both MLE and Rank Centrality on general comparison graphs.
Both the assumptions and guarantees are governed by the algebraic connectivity $\lambda_\cD$, which is a graph property directly computable from $\cD$.

For two designs $\cD_i=(([n],E_i),\{L_e^{(i)}\})$, $i\in\{1,2\}$, if $E_1\subseteq E_2$ and $L_e^{(1)}\le L_e^{(2)}$ for every $e\in E_1$, then $\bm{L}_{\cD_1}\preceq \bm{L}_{\cD_2}$ and therefore $\lambda_{\cD_1}\le\lambda_{\cD_2}$.
This implies that the error guarantees in \Cref{thm:general-entrywise} are graph-monotone in the sense that adding comparisons can only improve the bound.

\paragraph*{Minimal assumptions.}
The main assumption of \Cref{thm:general-entrywise} is \eqref{eq:general-baseline}, which is necessary for the stated error bounds to be $\widetilde O(1)$.
The dynamic range $\kappa$ is a standard parameter in the literature that captures how informative the comparison outcomes are, while the algebraic connectivity $\lambda_\cD$ measures how connected the comparison graph is.
In contrast, prior general-graph analyses typically require additional regularity conditions, such as a bounded degree ratio or a large normalized spectral gap, for the guarantees to be satisfactory \citep{NegahbanOhShah2017,LiShrotriyaRinaldo2022,YangChenOrecchiaMa2024,LeeMakurSingh2026}.

We illustrate the strength of our general-design guarantees in the following example.
\begin{example}[Heterogeneous random comparison graph]
  \label{ex:heterogeneous-comparison}
  Consider a sampling regime in which each edge $\{i,j\}$ is sampled independently with probability $p_{ij}\in[p,q]$, where $0<p\le q\le1$, and each sampled edge receives $L$ independent comparisons.
  Since $G(n,p)$ is a subgraph of this comparison graph, standard Erd\H{o}s--R\'enyi spectral concentration \citep[Lemma~6.4]{ChenFanMaWang2019} implies that $\lambda_{\cD}\gtrsim npL$ with high probability when $np\gtrsim\log n$.
  If $npL\gtrsim\kappa^2\log n$, \Cref{thm:general-entrywise} gives
  \begin{equation}
    \max_{a\in\{\MLE,\RC\}}
    \norm{\wh{\bm\theta}^{a}-\bm\theta^\star}_\infty
    \lesssim \frac{\kappa\log^{3/2}n}{\sqrt{npL}}.
    \label{eq:heterogeneous-entrywise}
  \end{equation}
  In the special case $p=q$, this recovers the familiar Erd\H{o}s--R\'enyi rate up to logarithmic factors \citep{ChenFanMaWang2019}.

  \Cref{tab:heterogeneous-entrywise} compares existing guarantees when applied to this graph family, including results for broader comparison models specialized to BTL.
  Here we take $\kappa=O(1)$ and $L=\Theta(1)$, and define the sampling heterogeneity ratio $\rho:=q/p$ and $a_n:=\sqrt{\log n/(np)}$.
  For all results, we assume $np\gtrsim\log n$, which is necessary for the graph to be connected with high probability.

\par\medskip\noindent
\begin{minipage}{\linewidth}
  \centering
  \captionsetup{hypcap=false}
  \captionof{table}{High-probability bounds on $\ell_\infty$ error in \Cref{ex:heterogeneous-comparison}.
  All rows assume $np\gtrsim\log n$.}
  \label{tab:heterogeneous-entrywise}
  \small
  \setlength{\tabcolsep}{4pt}
  \renewcommand{\arraystretch}{1.3}
  \begin{tabular*}{\linewidth}{@{\extracolsep{\fill}}llll@{}}
    \toprule
    Result & Estimator & Error bound & Sampling conditions \\
    \midrule
    \citet{LiShrotriyaRinaldo2022}, Thm.~1
      & Regularized MLE & $\dfrac{\rho}{p\sqrt n}+\sqrt{\rho}\,a_n$
      & $np^2\gtrsim\rho^2$, $np\gtrsim\rho\log n$ \\[2pt]
    \citet{HanXuChen2023}
      & MLE & $\rho\,a_n\log n$
      & $np\gg\rho^2\log^3 n$ \\
    \citet{YangChenOrecchiaMa2024}, Thm.~3
      & MLE & $a_n$
      & $np\gtrsim\rho^4\log^2 n$ \\
    \citet{YangChenOrecchiaMa2024}, Thm.~1
      & SDP-weighted MLE & $a_n$
      & $np\gtrsim\log^3 n$ \\
    \citet{FanLouWangYu2026}, Thm.~4.2
      & RC & $a_n$
      & $\rho=O(1)$, $np\gg\sqrt{n\log n}$ \\
    \citet{LeeMakurSingh2026}, Thm.~3.1
      & RC & $a_n$
      & $\rho=O(1)$ \\
    \citet{LeeMakurSingh2026}, Thm.~4.2
      & SDP-weighted RC & $a_n\bigl(1+\frac{\sqrt n}{np}\bigr)$
      & $np\gtrsim(n\log n)^{1/3}$ \\
    \citet{LiXu2026}, Thm.~4.1
      & MLE & $a_n+\dfrac{1}{p\sqrt n}$
      & $np\gtrsim\sqrt n\,\log^3 n$ \\[2pt]
    \textbf{Our work}, \Cref{thm:general-entrywise}
      & MLE and RC & $a_n\log n$
      & no additional assumption \\
    \bottomrule
  \end{tabular*}
\end{minipage}
\par\medskip

Our bound provides optimal error guarantees up to logarithmic factors for both MLE and Rank Centrality, with no additional assumptions.
In comparison, other existing work requires factors of $\rho=q/p$ in the error bounds \citep{LiShrotriyaRinaldo2022,HanXuChen2023} and sampling conditions \citep{LiShrotriyaRinaldo2022,HanXuChen2023,YangChenOrecchiaMa2024}, gives suboptimal error guarantees in sparse regimes \citep{LeeMakurSingh2026,LiXu2026}, or requires denser sampling regimes \citep{LiShrotriyaRinaldo2022,FanLouWangYu2026,LiXu2026}.
Moreover, it does not require the reweighting used in \citet{YangChenOrecchiaMa2024} and \citet{LeeMakurSingh2026}.
\end{example}

\paragraph*{Proof idea of Theorem~\ref{thm:general-entrywise}.}

For the MLE, let $\bm\delta=\wh{\bm\theta}^{\MLE}-\bm\theta^\star$, let $m$ be a median of its coordinates, and write $S_+(t)=\{i:\delta_i>m+t\}$ for $t\ge0$.
Our main technical step shows that, with high probability, for every $t\ge0$,
\begin{equation}
  \abs{S_+(t+a)}\le\frac14\abs{S_+(t)},
  \qquad
  a=A\kappa\sqrt{\frac{\log n}{\lambda_{\cD}}},
  \label{eq:proof-overview-contraction}
\end{equation}
where $A$ is a sufficiently large universal constant.
Each threshold increase by $a$ removes at least three quarters of the remaining coordinates.
Since $\abs{S_+(0)}\le n/2$, the set is empty after $O(\log n)$ steps.
The largest coordinate of $\bm\delta$ therefore lies at most $O(a\log n)$ above the median.
Repeating the argument below the median and using the centering constraint gives the entrywise rate $\kappa\log^{3/2}n/\sqrt{\lambda_{\cD}}$.

Rank Centrality follows the same principle, except that the level set and the contraction bound \eqref{eq:proof-overview-contraction} take a different form.
The shared tools and full proofs appear in \Cref{sec:shared-tools,sec:mle-proof,sec:rc-proof}.

\FloatBarrier

\subsection{Application: random-pairing Rank Centrality for the Rasch model}
\label{sec:rp-rc}

A key merit of \Cref{thm:general-entrywise} is that it accommodates general comparison designs, including those with nonuniform and dependent edge counts.
We illustrate this flexibility through the Rasch model, where the goal is to estimate item difficulties from binary responses by users of unknown abilities.

Consider $N$ users and $m\ge3$ items.
Each user--item pair $(t,i)$ is observed independently with probability $p$ and, when observed, produces an independent binary response
\[
  \Pp\{X_{ti}=1\}=\sig(\zeta_t^\star-\theta_i^\star),
\]
where $X_{ti}=1$ denotes a correct response, $\zeta_t^\star$ is user ability, and $\theta_i^\star$ is item difficulty.
We estimate $\bm\theta^\star$ under the constraint $\one_m^\top\bm\theta^\star=0$, treating the user abilities $\bm\zeta^\star$ as nuisance parameters.

To eliminate these nuisance parameters, \citet{YangMa2026Rasch} introduced a random-pairing reduction: for each user, draw a uniform random pairing of the observed items and retain pairs with different responses.
The identity
\begin{equation}
  \Pp\!\left\{
    X_{ti}=0,\ X_{tj}=1
    \mathrel{\big|}X_{ti}\ne X_{tj}
  \right\}
  =\sig(\theta_i^\star-\theta_j^\star)
  \label{eq:rp-rc-rasch-reduction}
\end{equation}
makes each retained pair a BTL comparison between item difficulties.
Let $\cD_Y$ denote the resulting item--item comparison design.
Its edge counts are dependent because each response is used in at most one pair.
Conditional on the pairings and retention events, however, the retained comparisons are independent BTL outcomes on $\cD_Y$ \citep[Section~4.1]{YangMa2026Rasch}.
Building on this reduction, we apply Rank Centrality to these comparisons and call the resulting estimator \emph{random-pairing Rank Centrality} (RP-RC).
Despite the dependence, our main theorem applies since it holds for general comparison designs.

\begin{corollary}[Random-pairing Rank Centrality]
  \label{cor:rp-rc}
Define the dynamic ranges
\[
  \kappa_1:=\exp\!\Big(
    \max_{i,j\in[m]}|\theta_i^\star-\theta_j^\star|
  \Big),\qquad
  \kappa_2:=\exp\!\Big(
    \max_{t\in[N],\,i\in[m]}|\zeta_t^\star-\theta_i^\star|
  \Big).
\]
Suppose $mp\ge2$, $m\le N^{100}$, and
\begin{equation}
  Np\ge C\left(
    \kappa_2^2\log N+\kappa_1^2\kappa_2\log m
  \right).
  \label{eq:rp-rc-baseline}
\end{equation}
Then, with probability at least $1-Cm^{-10}-CN^{-10}$, the RP-RC estimator exists uniquely and satisfies
\begin{equation}
  \norm{\wh{\bm\theta}^{\mathrm{RP\text{-}RC}}-\bm\theta^\star}_\infty
  \le C\kappa_1\sqrt{\frac{\kappa_2\log^3m}{Np}}.
  \label{eq:rp-rc-rate}
\end{equation}
\end{corollary}

\begin{proof}
  By \citet[Lemma~3]{YangMa2026Rasch}, the first term in the sampling condition \eqref{eq:rp-rc-baseline} gives $\lambda_{\cD_Y}\ge Np/(4\kappa_2)$ with probability at least $1-CN^{-10}$.
  The second term then verifies \eqref{eq:general-baseline}, so conditioning on the realized reduction and applying \Cref{thm:general-entrywise} yields the claim by a union bound.
\end{proof}

RP-RC matches the sampling requirements and entrywise guarantee of RP-MLE \citep{YangMa2026Rasch} up to logarithmic factors.
In comparison, the existing spectral method requires the denser sampling conditions $Np^2\gtrsim\log m$ and $mp\gtrsim\log N$ for a comparable entrywise guarantee \citep[Theorem~3.1]{NguyenZhang2023}.

%% file: sections/04_outcome_adaptive_augmentation.tex
\section{Outcome-adaptive augmentation}
\label{sec:random-design-models}

The main results in \Cref{thm:general-entrywise} provide guarantees for both MLE and Rank Centrality in terms of the algebraic connectivity $\lambda_{\cD}$.
An important observation is that $\lambda_{\cD}$ can only increase when comparisons are added.
This raises a natural robustness question: do the graph-monotone general-design guarantees still hold when the choice of additional comparisons depends on the initial outcomes?

Consider the outcome-adaptive augmentation of a general initial comparison design defined as follows.
\begin{definition}[Outcome-adaptive augmentation]
  \label{def:semirandom}
Fix an initial design $\cD_0=(G_0,\{L_{ij}^{(0)}\})$, where $G_0=([n],E_0)$ and $L_{ij}^{(0)}$ is a positive integer for every $\{i,j\}\in E_0$ and zero otherwise.
\begin{enumerate}[label=(\roman*)]
  \item \emph{Initial outcomes.}
  Draw $L_{ij}^{(0)}$ independent BTL comparisons on every edge $\{i,j\}\in E_0$.

  \item \emph{Adversarial augmentation.}
  The adversary observes $\cD_0$, the initial outcomes, and possibly $\bm\theta^\star$.
  It then chooses nonnegative integers $M_{ij}$ for all pairs $\{i,j\}$, which may be positive for pairs outside $E_0$.
  Conditional on the initial data and these choices, $M_{ij}$ fresh independent BTL comparisons are drawn on each pair $\{i,j\}$.

  \item \emph{Final observations.}
  Let $\cD$ be the final design, with counts $L_{ij}=L_{ij}^{(0)}+M_{ij}$.
  The estimators observe all comparisons together, without knowing which were collected initially.
\end{enumerate}
\end{definition}

A special case of \Cref{def:semirandom} is the heterogeneous sampling model in \Cref{ex:heterogeneous-comparison}.\footnote{See \Cref{prop:heterogeneous-is-semirandom} for a formal justification.}
More generally, \Cref{def:semirandom} allows the adversary to choose additional comparisons based on the initial outcomes and possibly $\bm\theta^\star$.
The resulting dependence between the final design and the initial outcomes prevents a direct application of \Cref{thm:general-entrywise}.

We adapt the general-design analysis to outcome-adaptive augmentation to obtain the following theorem.
It gives an error bound that is governed by the algebraic connectivity of the final design.
Since algebraic connectivity can only increase with additional comparisons, this shows that both MLE and Rank Centrality are robust to outcome-adaptive augmentation. 
\begin{theorem}[Outcome-adaptive augmentation]
  \label{thm:semirandom}
Under \Cref{def:semirandom}, fix $n\ge3$ and $\bm\theta^\star\in\cH$ with dynamic range $\kappa$.
There are sufficiently large universal constants $C,C_{\mathrm{inf}}>0$ such that if
\begin{equation}
  \lambda_{\cD_0}\ge C\kappa^2\log n,
  \label{eq:semirandom-baseline}
\end{equation}
then, uniformly over the adversary, with probability at least $1-Cn^{-10}$, the MLE and Rank Centrality estimators exist uniquely and satisfy
\[
  \max_{a\in\{\MLE,\RC\}}
  \norm{\wh{\bm\theta}^{a}-\bm\theta^\star}_\infty
  \le C_{\mathrm{inf}}\kappa
  \frac{\log^{3/2}n}{\sqrt{\lambda_{\cD}}}.
\]
\end{theorem}

The proof handles the dependence introduced by outcome-adaptive augmentation in two stages.
We first control the noise from the initial outcomes with $\cD_0$ fixed, and then condition on the initial data and chosen counts to control the noise from the fresh outcomes.
The full proof is deferred to \Cref{sec:semirandom-theorem-proof}.

\paragraph*{Semirandom model.}
Another special case of \Cref{def:semirandom} is the semirandom model, where $G_0\sim G(n,p)$ and $L_{ij}^{(0)}=L$ for every $\{i,j\}\in E_0$.
When $np\ge C\log n$ and $npL\ge C\kappa^2\log n$, $\lambda_\cD\ge \lambda_{\cD_0}\gtrsim npL$ (see \Cref{lem:er-gap}).
\Cref{thm:semirandom} then gives $\ell_\infty$ error at most $C\kappa\log^{3/2}n/\sqrt{npL}$ for both estimators, which matches the minimax Erd\H{o}s--R\'enyi rates up to logarithmic factors \citep{ChenFanMaWang2019,ChenGaoZhang2022}.

This semirandom guarantee complements two recent approaches based on reweighting or additional design regularity.
\citet{YangChenOrecchiaMa2024} obtain a near-optimal guarantee for an SDP-reweighted MLE when the adversary chooses the augmentation observing initial comparison graph but not the initial outcomes.
\citet{LeeMakurSingh2026} study Rank Centrality under heterogeneous sampling: their unweighted guarantee requires additional regularity and a favorable normalized spectral gap, whereas their weighted guarantee is suboptimal in the sparse regime $np\ll\sqrt n$.
In comparison with these prior results, \Cref{thm:semirandom} applies to the unweighted MLE and unweighted Rank Centrality without additional design regularity or a normalized-gap assumption.
It also applies to the case where the adversary can observe the initial outcomes and the latent parameter.

%% file: sections/05_experiments.tex
\section{Empirical and numerical experiments}
\label{sec:experiments}

\subsection{LMArena replay}
\label{sec:experiments-arena}

Our main results in \Cref{thm:general-entrywise} show that the entrywise error guarantee depends more on the geometry of the comparison design than on the number of comparisons involving each item.
We illustrate this phenomenon in a real-data study using the LMArena human-preference dataset \citep{ChiangEtAl2024Arena,LMArenaHumanPreference140k}, which contains pairwise comparisons of large language models (LLMs) collected from human judges.

We take the 30 models that have the most comparisons and perform a replay experiment in which we resample two comparison graphs with the same number of comparisons but different geometries:
\begin{itemize}
  \item \emph{Local repetition.}
  Start with a random $10$-regular graph.
  Randomly partition the models once into ten groups of three and then add five comparisons to each of the $30$ within-group pairs at every step.
  \item \emph{Global linking.}
  Start with a random $10$-regular graph.
  Add ten new random perfect matchings at every step.
\end{itemize}
  At step $s\in\{1,\ldots,10\}$, both graphs are regular of degree $10+10s$.

In each trial, we set aside $20\%$ of the votes on every pair to fit reference estimates $\wh{\bm\theta}^{\MLE,\mathrm{ref}}$ and $\wh{\bm\theta}^{\RC,\mathrm{ref}}$.
We then resample the comparison design from the remaining $80\%$ and compute the MLE $\wh{\bm\theta}^{\MLE}$ and the Rank Centrality estimate $\wh{\bm\theta}^{\RC}$ from the resampled data.
We compute the maximum entrywise difference between the resampled and reference estimates:
\begin{equation}
  R_a
  =\norm{\wh{\bm\theta}^{a}-\wh{\bm\theta}^{a,\mathrm{ref}}}_\infty,
  \qquad a\in\{\MLE,\RC\}.
  \label{eq:arena-reference-distance}
\end{equation}

\begin{figure}[!htbp]
  \centering
  \includegraphics[width=0.98\linewidth]{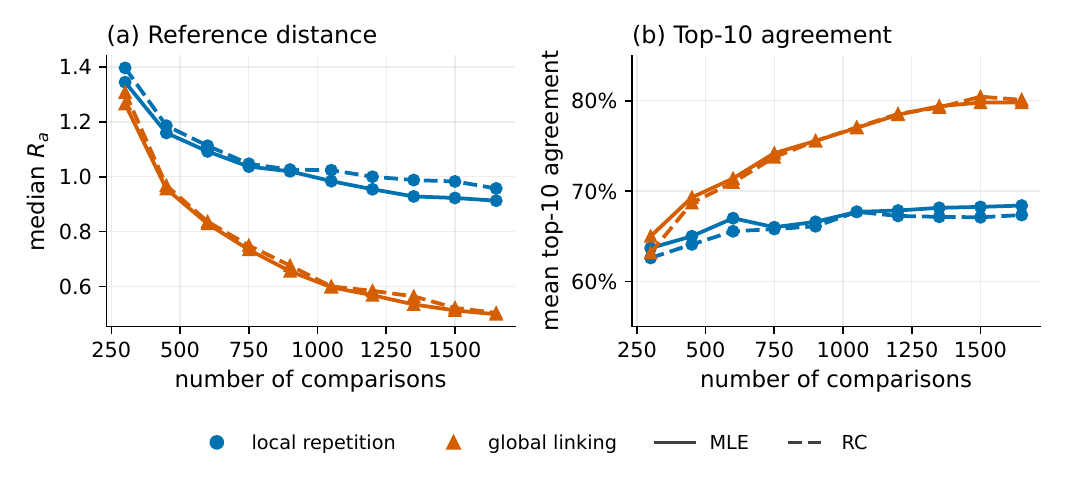}
  \caption{Chatbot Arena replay over $200$ trials, shown for budgets from $300$ to $1{,}650$ comparisons.
  \textbf{(a)} Median $R_a$ against the number of comparisons.
  \textbf{(b)} Mean top-$10$ agreement with the held-out reference against the number of comparisons.}
  \label{fig:arena-degree-matched}
\end{figure}

\Cref{fig:arena-degree-matched} shows that the MLE and Rank Centrality perform similarly.
Both perform better with global linking than with local repetition, despite the fact that both graphs are regular with the same degree.
This finding is consistent with our theory, which suggests that the geometry of the comparison graph matters more than the number of comparisons each item receives.
For instance, at the rightmost point in \Cref{fig:arena-degree-matched}, each graph consists of $1{,}650$ comparisons, but the median algebraic connectivity is $7.973$ for local repetition versus $95.006$ for global linking.

\subsection{Complete multipartite comparison graphs}
\label{sec:experiments-multipartite}
We next compare the performance of MLE on a group of designs to illustrate how the error scales with the algebraic connectivity.
Let $K_{n_1,\ldots,n_r}$ denote the complete (multipartite) graph with part sizes $n_1,\ldots,n_r$.
We consider three comparison graphs on the same $64$ items: $K_{32,32}$, $K_{16,16,16,16}$, and $K_{64}$.
We draw $L$ independent BTL comparisons on every edge of each graph and fit the MLE on the resulting outcomes. The corresponding algebraic connectivities are $\lambda_{\cD}=32L$, $48L$, and $64L$, respectively.

\begin{figure}[!htbp]
  \centering
  \includegraphics[width=0.96\linewidth]{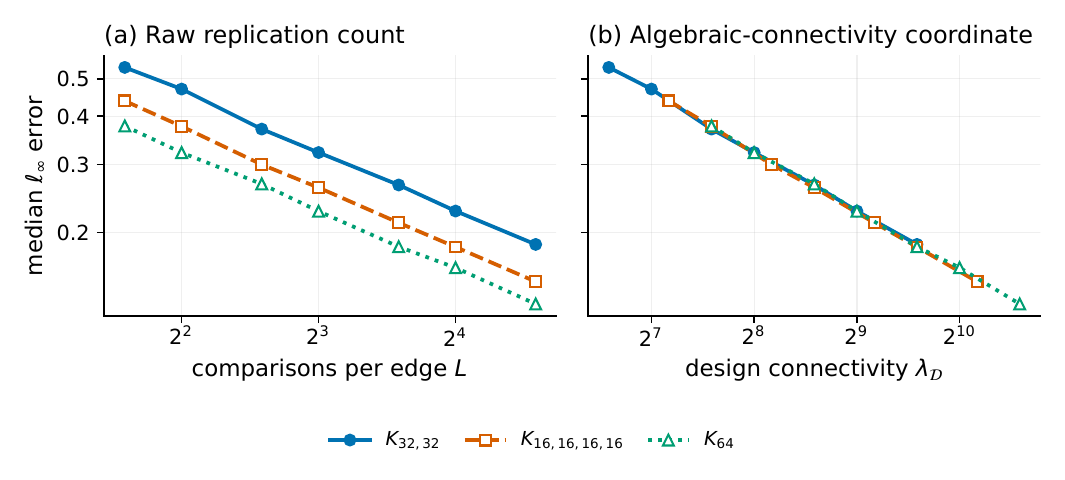}
  \caption{Median MLE entrywise error over $300$ paired trials on $K_{32,32}$, $K_{16,16,16,16}$, and $K_{64}$, plotted against \textbf{(a)} comparisons per edge and \textbf{(b)} the algebraic connectivity of the observation Laplacian. Both panels use log--log scales.}
  \label{fig:multipartite-experiment}
\end{figure}

\Cref{fig:multipartite-experiment} shows MLE has a different entrywise error for each of the three designs at the same number of comparisons per edge.
However, once we compare the error against $\lambda_\cD$, the three MLE curves collapse onto a common $1/\sqrt{\lambda_{\cD}}$ scaling.
This is consistent with the dependence on algebraic connectivity in \Cref{thm:general-entrywise}.

\subsection{Semirandom sampling}
\label{sec:experiments-semirandom}

Consider the semirandom model introduced in \Cref{sec:random-design-models}. 
We now compare the unweighted MLE and Rank Centrality under two designs: the initial Erd\H{o}s--R\'enyi graph $G_0\sim G(n,p_n)$ and a semirandom graph that augments $G_0$ by adding all missing edges among the bottom third of the latent ranking.
We also fit the SDP-weighted MLE of \citet{YangChenOrecchiaMa2024} on the augmented graph.

\begin{figure}[!htbp]
  \centering
  \includegraphics[width=0.98\linewidth]{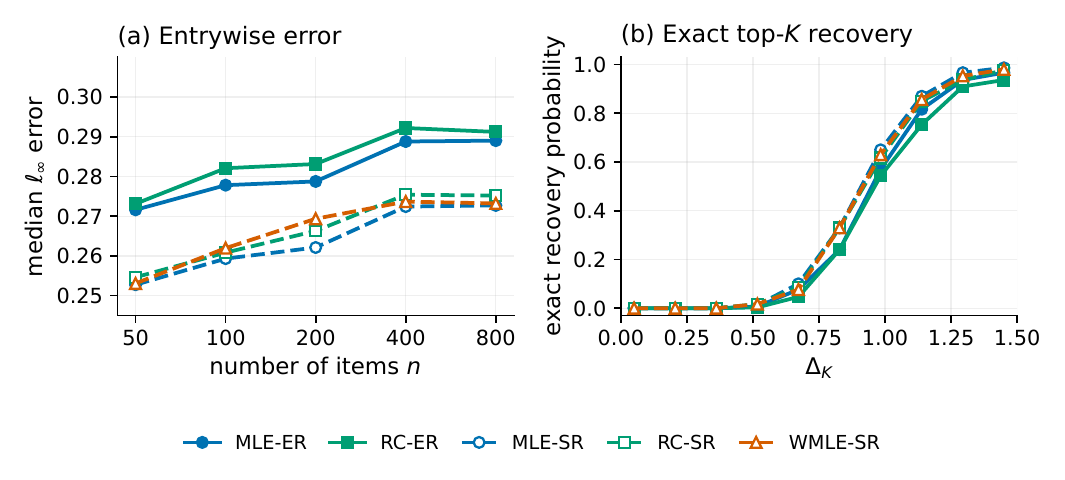}
  \caption{Effect of monotone augmentation for MLE, Rank Centrality (RC), and SDP-weighted MLE (WMLE) \citep{YangChenOrecchiaMa2024}, over $300$ paired trials per parameter set.
  ER refers to the Erd\H{o}s--R\'enyi base graph and SR refers to the augmented semirandom graph.
  \textbf{(a)} Median entrywise errors versus $n$, with $p_n=4\log(n)/n$ and $L=25$.
  \textbf{(b)} Top-$K$ recovery probability versus $\Delta_K$, with $n=100$, $p=0.2$, $K=10$, and $L=4$.}
  \label{fig:experiments}
\end{figure}

\Cref{fig:experiments} illustrates the robustness of both unweighted MLE and RC to semirandom augmentation.
It shows that augmentation reduces the median entrywise errors and increases the top-$K$ recovery probabilities for both MLE and RC.
On the augmented graph, both unweighted estimators perform similarly to the SDP-weighted MLE of \citet{YangChenOrecchiaMa2024}.
These observations are consistent with our theoretical results in \Cref{sec:random-design-models}, which apply to both unweighted MLE and RC without requiring reweighting.

%% file: sections/06_discussion.tex
\section{Discussion}
\label{sec:discussion}

This paper establishes graph-monotone entrywise guarantees for the MLE and Rank Centrality in the BTL model on general comparison graphs.
Both estimators achieve an entrywise error rate of $\widetilde O(1/\sqrt{\lambda_{\cD}})$ and are robust to outcome-adaptive augmentation.

A natural direction for future work is to build on the augmentation robustness of the MLE and Rank Centrality.
Our result is mostly concerned with robustness against potentially adversarial augmentation.
It would be interesting to leverage this robustness to design adaptive sample-augmentation strategies that improve estimation accuracy.

Our analysis is tailored to the maximum entrywise error over all pairs. 
Another direction for future work is to consider estimation error of the score difference $\theta_i^\star-\theta_j^\star$ for any given pair ${i,j}$ and see how this error depends on the design geometry and interacts with data augmentation. 
We leave both questions for future work.

%% file: sections/a_shared_tools.tex
\section{Shared probabilistic and graph tools}
\label{sec:shared-tools}

This appendix collects three ingredients shared by the MLE and Rank Centrality proofs.
In Sec.~A.1, we control the localized edge-noise process uniformly against sparse test vectors.
In Sec.~A.2, we use the algebraic connectivity of the design Laplacian to convert Laplacian energy into Euclidean control.
In Sec.~A.3, we use the same algebraic connectivity to lower bound the comparison mass across every cut and prove strong connectivity of the observed-win digraph.

\subsection{Uniform concentration over sparse vectors}

\begin{lemma}[Localized edge-noise process]
  \label{lem:localized-process}
  Fix a connected observation design $\cD=(G,\{L_e\}_{e\in E})$, where $G=([n],E)$.
  Let $\{\zeta_e:e\in E\}$ be independent, centered random variables satisfying
  \begin{equation}
    \E\exp(t\zeta_e)\le\exp(\tau^2L_et^2/2),
    \qquad t\in\R,
    \label{eq:edge-mgf}
  \end{equation}
  and put $\bm z=\sum_{e\in E}\zeta_e\bm b_e$.
  There is a universal constant $C>0$ such that, with probability at least $1-Cn^{-10}$, simultaneously for every nonempty $S\subsetneq[n]$ with $s=|S|\le n/2$ and every $\bm v\in\R^n$ supported on $S$,
  \begin{equation}
    \abs{\bm z^\top\bm v}
    \le C\tau\sqrt{s\log n}
    \sqrt{\bm v^\top\bm L_{\cD}\bm v}.
    \label{eq:localized-process}
  \end{equation}
\end{lemma}

\begin{proof}
  For a fixed $\bm v$, independence and \eqref{eq:edge-mgf} give
  \[
    \E\exp(t\bm z^\top\bm v) \le \exp\!
    \left(\frac{\tau^2t^2}{2}
    \sum_{e\in E}
    L_e\cdot(\bm b_e^\top\bm v)^2\right) =\exp\!
    \left(\frac{\tau^2t^2}{2}
    \bm v^\top\bm L_{\cD}\bm v\right).
  \]
  Consequently, applying the Chernoff bound separately to $\bm z^\top\bm v$ and $-\bm z^\top\bm v$ gives
  \begin{equation}
    \Pp\!\left\{
    \abs{\bm z^\top\bm v}>
    u\tau\sqrt{\bm v^\top\bm L_{\cD}\bm v}
    \right\}
    \le2\exp(-u^2/2).
    \label{eq:fixed-vector-tail}
  \end{equation}

  Fix $S\subsetneq[n]$ and define
  \[
    \norm{\bm v}_{\cD} :=(\bm v^\top\bm L_{\cD}\bm v)^{1/2}, \qquad \cV_S:=\{\bm v:\supp(\bm v)\subseteq S\}.
  \]
  Connectedness implies that the only null vectors of $\bm L_{\cD}$ are multiples of $\one_n$.
  Hence $\norm{\cdot}_{\cD}$ is a norm on the $s$-dimensional space $\cV_S$.
  By a standard volumetric covering argument, its unit sphere admits a $1/2$-net $\mathcal N_S$ in this norm with $|\mathcal N_S|\le5^s$.

  Set
  \[
    M_S:=\sup_{\substack{\bm v\in\cV_S\\\norm{\bm v}_{\cD}=1}} \abs{\bm z^\top\bm v}.
  \]
  For every unit vector $\bm v$, choose $\bm q\in\mathcal N_S$ with $\norm{\bm v-\bm q}_{\cD}\le1/2$.
  The triangle inequality and the definition of $M_S$ then give
  \[
    \abs{\bm z^\top\bm v} \le \abs{\bm z^\top(\bm v-\bm q)} + \abs{\bm z^\top\bm q} \le \frac12M_S + \abs{\bm z^\top\bm q}.
  \]
  Taking the supremum over $\bm v$ and rearranging yields
  \[
    M_S = \sup_{\substack{\bm v\in\cV_S\\\norm{\bm v}_{\cD}=1}}
    \abs{\bm z^\top\bm v}
    \le2\max_{\bm q\in\mathcal N_S}\abs{\bm z^\top\bm q}.
  \]
  Consequently, for every $u>0$,
  \[
    \{M_S>2u\tau\}
    \subseteq
    \bigcup_{\bm q\in\mathcal N_S}
    \left\{\abs{\bm z^\top\bm q}>u\tau\right\}.
  \]
  Apply \eqref{eq:fixed-vector-tail} to the net $\mathcal N_S$ and take a union bound over the at most $\binom ns\le n^s$ supports of size $s$.
  With $u=A\sqrt{s\log n}$, the probability that $M_S>2u\tau$ for at least one support of size at most $n/2$ is bounded by
  \[
    2\sum_{s=1}^{\lfloor n/2\rfloor} \binom ns\cdot 5^s\cdot\exp\left(-\frac{1}{2}A^2s\log n\right) \le 2\sum_{s=1}^{\lfloor n/2\rfloor} \exp\!
    \left\{s\log(5n)-\frac{1}{2}
    A^2s\log n\right\} \le Cn^{-10}
  \]
  for a sufficiently large universal constant $A$.
  On the complementary event, $M_S\le2A\tau\sqrt{s\log n}$.
  Hence, for every nonzero $\bm v\in\cV_S$,
  \[
    \abs{\bm z^\top\bm v} =\norm{\bm v}_{\cD} \abs{\bm z^\top\frac{\bm v}{\norm{\bm v}_{\cD}}} \le2A\tau\sqrt{s\log n}\,\norm{\bm v}_{\cD},
  \]
  which proves \eqref{eq:localized-process} after absorbing $2A$ into $C$.
\end{proof}

\subsection{From Laplacian energy to Euclidean control}

\begin{lemma}[Sparse-support Poincar\'e inequality]
  \label{lem:sparse-poincare}
  For any fixed connected observation design $\cD$, every $\bm v\in\R^n$ supported on a set of size at most $n/2$ satisfies
  \begin{equation}
    \bm v^\top\bm L_{\cD}\bm v
    \ge\frac{\lambda_{\cD}}{2}\norm{\bm v}_2^2.
    \label{eq:sparse-poincare}
  \end{equation}
\end{lemma}

\begin{proof}
  Let $\overline v=n^{-1}\one_n^\top\bm v$.
  Since $\bm v-\overline v\one_n\perp\one_n$, the standard variational characterization of $\lambda_{\cD}=\lambda_2(\bm L_{\cD})$ gives
  \begin{equation}
    \bm v^\top\bm L_{\cD}\bm v
    =(\bm v-\overline v\one_n)^\top\bm L_{\cD}
    (\bm v-\overline v\one_n)
    \ge\lambda_{\cD}\norm{\bm v-\overline v\one_n}_2^2.
    \label{eq:sparse-poincare-variational}
  \end{equation}
  If $S=\supp(\bm v)$ and $|S|\le n/2$, then Cauchy--Schwarz gives
  \[
    n\overline v^2 =\frac{(\sum_{i\in S}v_i)^2}{n} \le\frac{|S|}{n}\norm{\bm v}_2^2 \le\frac12\norm{\bm v}_2^2.
  \]
  Consequently, $\norm{\bm v-\overline v\one_n}_2^2=\norm{\bm v}_2^2-n\overline v^2 \ge\norm{\bm v}_2^2/2$.
  Combining this with \eqref{eq:sparse-poincare-variational} finishes the proof.
\end{proof}

\subsection{Strong connectivity of the observed win graph}

For $S\subset[n]$, let $\partial S:=\{\{i,j\}\in E:\abs{\{i,j\}\cap S}=1\}$ denote its edge boundary.

\begin{lemma}[Cut weight from algebraic connectivity]
  \label{lem:spectral-cut-size}
  For any fixed connected observation design $\cD$ and every nonempty $S\subsetneq[n]$,
  \begin{equation}
    \sum_{e\in\partial S}
    L_e \ge \lambda_{\cD}\frac{|S|(n-|S|)}{n}.
    \label{eq:spectral-cut-size}
  \end{equation}
\end{lemma}

\begin{proof}
  Let $\bm x=\one_S-(|S|/n)\one_n$.
  Then $\bm x\perp\one_n$, $\bm x^\top\bm L_{\cD}\bm x=\sum_{e\in\partial S}L_e$, and $\norm{\bm x}_2^2=|S|(n-|S|)/n$.
  The variational definition of $\lambda_{\cD}$ proves the claim.
\end{proof}

Recall that a directed graph is strongly connected if, for every ordered pair of vertices $(i,j)$, it contains a directed path from $i$ to $j$.
We form the directed graph $D_{\mathrm{win}}$ by connecting $i\to j$ whenever item $j$ defeats item $i$ at least once on the observed pair.

\begin{lemma}[Strong connectivity for a fixed observation design]
  \label{lem:general-directed-connectivity}
  For a fixed connected observation design $\cD$, if
  \begin{equation}
    \frac{\lambda_{\cD}}{1+\kappa}\ge C\log n,
    \label{eq:connectivity-baseline}
  \end{equation}
  then $D_{\mathrm{win}}$ is strongly connected with probability at least $1-Cn^{-10}$.
\end{lemma}

\begin{proof}
  Fix a nonempty $S\subsetneq[n]$.
  For every comparison across its edge boundary, the probability that the endpoint outside $S$ defeats the endpoint in $S$ is at least $1/(1+\kappa)$.
  For $D_\mathrm{win}$, such an outcome creates a directed edge from $S$ to $S^c$.
  Independence across edges and repetitions gives
  \[
    \Pp\{\text{no directed edge of $D_{\mathrm{win}}$ leaves }S\} \le \exp\!
    \left\{-\frac{\sum_{e\in\partial S}
      L_e}{1+\kappa}\right\}.
  \]
  A directed graph on $[n]$ is strongly connected if and only if every nonempty $S\subsetneq[n]$ has a directed edge from $S$ to $S^c$.
  Applying \Cref{lem:spectral-cut-size} to every such $S$ and then using $n-s\ge n/2$ for $s\le n/2$ yields
  \begin{align*}
    \Pp\{D_{\mathrm{win}}\text{ is not strongly connected}\}
     & \le
    2\sum_{s=1}^{\lfloor n/2\rfloor}
    \binom ns
    \exp\!\left\{-\frac{\lambda_{\cD}s(n-s)}{n(1+\kappa)}\right\} \\
     & \le
    2\sum_{s=1}^{\lfloor n/2\rfloor}
    \exp\!\left\{s\log n-\frac{\lambda_{\cD}s}{2(1+\kappa)}\right\}
    \le Cn^{-10}.
  \end{align*}
  The last line follows from \eqref{eq:connectivity-baseline}.
\end{proof}

%% file: sections/b_mle_proof.tex
\section{Proof of Theorem~\ref{thm:general-entrywise}: MLE}
\label{sec:mle-proof}

This appendix proves the MLE part of \Cref{thm:general-entrywise}.
We first establish existence and localized score control, then derive a truncation-energy bound from the logistic increment.
Finally, we use an additive level-set contraction argument to obtain entrywise control.

\subsection{Existence, score identity, and localization}

Recall that $D_{\mathrm{win}}$ is the directed graph with $i\to j$ whenever item $j$ defeats item $i$ at least once.

We say that a loss function is \emph{coercive} if its value tends to infinity whenever the Euclidean norm of its argument tends to infinity.

\begin{lemma}[Strong connectivity implies a finite unique MLE]
  \label{lem:mle-existence}
  If $D_{\mathrm{win}}$ is strongly connected, the loss \eqref{eq:mle-loss} is coercive and strictly convex on $\cH$.
  Hence the MLE exists uniquely.
\end{lemma}

\begin{proof}
  Let $N_{ij}$ be the number of observed wins of item $i$ over item $j$.
  Equivalently, grouping the terms in \eqref{eq:mle-loss} by their observed winner gives
  \begin{equation}
    \cL(\bm\theta)
    =\sum_{i\ne j}
    N_{ij} \log\!
    \left(1+\e^{\theta_j-\theta_i}\right).
    \label{eq:ordered-win-loss}
  \end{equation}
  Its Hessian is
  \[
    \nabla^2\cL(\bm\theta) =\sum_{e\in E}L_e\cdot\sig'(\bm b_e^\top\bm\theta) \bm b_e\bm b_e^\top.
  \]
  Every coefficient is positive.
  Since the comparison graph $G$ is connected, the Hessian's nullspace is $\operatorname{span}\{\one_n\}$ and the loss is strictly convex on $\cH$.

  For coercivity, take any sequence $\bm\theta^{(m)}\in\cH$ with $\norm{\bm\theta^{(m)}}_2\to\infty$.
  Its coordinate range $R_m=\max_i\theta_i^{(m)}-\min_i\theta_i^{(m)}$ tends to infinity.
  Among the ordered coordinates, some consecutive gap is at least $R_m/(n-1)$.
  Let $S_m$ and $[n]\setminus S_m$ be the sets of items separated by this gap, with $\theta_i^{(m)}<\theta_j^{(m)}$ for all $i\in S_m$ and $j\in[n]\setminus S_m$.
  Strong connectivity supplies a directed edge $j\to i$ of $D_{\mathrm{win}}$ with $i\in S_m$ and $j\notin S_m$.
  Thus item $i$ defeats item $j$ at least once.
  Its contribution to \eqref{eq:ordered-win-loss} is at least
  \[
    \log\!
    \left(1+\e^{\theta_j^{(m)}-\theta_i^{(m)}}\right)
    \ge\frac{R_m}{n-1}.
  \]
  Thus the loss is at least as large as this term and is coercive because $R_m\to\infty$.
  Continuity, coercivity, and strict convexity of $\cL$ on $\cH$ yield a unique minimizer.
\end{proof}

\begin{lemma}[MLE score identity and localized score control]
  \label{lem:mle-score-localization}
  Define the score at the truth by
  \begin{equation}
    \bm g_{\MLE}:=\nabla\cL(\bm\theta^\star)
    =\sum_{e\in E}
    L_e\cdot \bigl(\sig(\bm b_e^\top\bm\theta^\star)-\overline Y_e\bigr)\bm b_e.
    \label{eq:mle-score}
  \end{equation}
  Whenever the MLE exists, set
  \[
    \bm\delta:=\wh{\bm\theta}^{\MLE}-\bm\theta^\star.
  \]
  Then
  \begin{equation}
    \bm 0=\bm g_{\MLE}+\sum_{e\in E}
    L_e\cdot\left\{\sig\bigl(\bm b_e^\top(\bm\theta^\star+\bm\delta)\bigr) -\sig(\bm b_e^\top\bm\theta^\star)\right\}\bm b_e.
    \label{eq:mle-score-identity}
  \end{equation}
  Condition on the observation design $\cD$.
  Then, with probability at least $1-Cn^{-10}$,
  \begin{equation}
    \abs{\bm g_{\MLE}^\top\bm v}
    \le C\sqrt{s\log n}
    \sqrt{\bm v^\top\bm L_{\cD}\bm v}
    \label{eq:mle-localized-score}
  \end{equation}
  simultaneously for every $\bm v$ supported on a set of size $s\le n/2$.
\end{lemma}

\begin{proof}
  At a constrained minimizer on $\cH$, the gradient is orthogonal to $\cH$ and therefore belongs to $\operatorname{span}\{\one_n\}$.
  On the other hand, $\nabla\cL(\bm\theta)\perp\one_n$ for every $\bm\theta$, because the loss is invariant under common shifts.
  The constrained first-order condition thus gives $\nabla\cL(\wh{\bm\theta}^{\MLE})=0$.
  Subtracting $\nabla\cL(\bm\theta^\star)=\bm g_{\MLE}$ and expanding the gradient difference gives \eqref{eq:mle-score-identity}.

  Condition on the observation design $\cD$.
  The variables $\zeta_e=L_e\cdot\{\sig(\bm b_e^\top\bm\theta^\star)-\overline Y_e\}$ are independent and centered.
  Hoeffding's lemma gives $\E\exp(t\zeta_e)\le\exp(L_et^2/8)$.
  Applying \Cref{lem:localized-process} with a universal $\tau$ proves \eqref{eq:mle-localized-score}.
\end{proof}

\subsection{Truncated logistic increment}

For $t\in\R$ and $a\in(0,1]$, define
\begin{equation}
  T_{t,a}(z):=\min\{(z-t)_+,a\}.
  \label{eq:mle-truncation}
\end{equation}

\begin{lemma}[Logistic increment lower bound]
  \label{lem:logistic-increment}
  For every $\kappa\ge1$ and $x,u\in\R$ with $|x|\le\log\kappa$,
  \begin{equation}
    \abs{\sig(x+u)-\sig(x)}
    \ge\frac{1}{2(\e+1)\kappa}\min\{|u|,1\}
    \ge\frac{1}{8\kappa}\min\{|u|,1\}.
    \label{eq:logistic-increment}
  \end{equation}
\end{lemma}

\begin{proof}
  For $|z|\le\log\kappa+1$,
  \[
    \sig'(z)=\frac1{\e^{-z}+2+\e^z}\ge\frac1{2(\e+1)\kappa}.
  \]
  If $|u|\le1$, the mean value theorem therefore gives
  \[
    |\sig(x+u)-\sig(x)|\ge \frac{1}{2(\e+1)\kappa}|u|.
  \]
  If $u>1$, monotonicity of $\sig$ gives
  \[
    |\sig(x+u)-\sig(x)|\ge |\sig(x+1)-\sig(x)| \ge \frac{1}{2(\e+1)\kappa}.
  \]
  If $u<-1$, the analogous comparison in the opposite direction gives
  \[
    |\sig(x+u)-\sig(x)|\ge |\sig(x-1)-\sig(x)| \ge \frac{1}{2(\e+1)\kappa}.
  \]
  This proves the first inequality in \eqref{eq:logistic-increment}, and the second follows from $\e<3$.
\end{proof}

\subsection{Truncated level-set energy}

\begin{lemma}[MLE truncated level-set energy]
  \label{lem:mle-truncated-energy}
  Assume that \eqref{eq:mle-localized-score} in \Cref{lem:mle-score-localization} holds.
  Fix $t\in\R$ and $a\in(0,1]$, and define $\bm v$ by
  \[
    v_i:=T_{t,a}(\delta_i), \qquad i\in[n].
  \]
  Let $S:=\{i:\delta_i>t\}$.
  If $1\le |S|\le n/2$, then
  \begin{align}
    \bm v^\top\bm L_{\cD}\bm v
     & \le C\kappa^2|S|\log n,
    \label{eq:mle-energy}                            \\
    \norm{\bm v}_2^2
     & \le C\frac{\kappa^2}{\lambda_{\cD}}|S|\log n.
    \label{eq:mle-l2-truncation}
  \end{align}
  For the lower truncation, define
  \[
    \wt T_{t,a}(z):=-\min\{(t-z)_+,a\}, \qquad v_i:=\wt T_{t,a}(\delta_i),
  \]
  and let $S^-:=\{i:\delta_i<t\}$.
  If $1\le |S^-|\le n/2$, the same bounds hold with $|S|$ replaced by $|S^-|$.
\end{lemma}

\begin{proof}
  For an oriented edge $e=(i,j)$, put $u=\bm b_e^\top\bm\delta = \delta_i-\delta_j$.
  Because $T_{t,a}$ is nondecreasing, $1$-Lipschitz, and takes values in $[0,a]\subset[0,1]$, $\bm b_e^\top\bm v=T_{t,a}(\delta_i)-T_{t,a}(\delta_j)$ has the same sign as $u$ and $|\bm b_e^\top\bm v|\le\min\{|u|,1\}$.
  Since $|\bm b_e^\top\bm\theta^\star|\le\log\kappa$, \Cref{lem:logistic-increment} gives
  \[
    \left[\sig\bigl(\bm b_e^\top(\bm\theta^\star+\bm\delta)\bigr) -\sig(\bm b_e^\top\bm\theta^\star)\right] \cdot(\bm b_e^\top\bm v) \ge \frac{1}{8\kappa}\min\{|u|,1\} \cdot \abs{\bm b_e^\top\bm v} \ge\frac{1}{8\kappa}(\bm b_e^\top\bm v)^2.
  \]
  Taking the inner product of the score identity \eqref{eq:mle-score-identity} in \Cref{lem:mle-score-localization} with $\bm v$ and summing the preceding inequality over the edges yields
  \[
    -\bm g_{\MLE}^\top\bm v \ge\frac{1}{8\kappa} \sum_{e\in E}L_e\cdot(\bm b_e^\top\bm v)^2 =\frac{1}{8\kappa}\bm v^\top\bm L_{\cD}\bm v.
  \]
  Combining this with the localized-score bound in \Cref{lem:mle-score-localization} gives
  \[
    \frac{1}{8\kappa}\bm v^\top\bm L_{\cD}\bm v
    \le C\sqrt{|S|\log n}
    \sqrt{\bm v^\top\bm L_{\cD}\bm v}.
  \]
  Rearranging this proves \eqref{eq:mle-energy}.
  Combining it with \Cref{lem:sparse-poincare} proves \eqref{eq:mle-l2-truncation}.
  The lower truncation works similarly.
\end{proof}

\subsection{Completion of the MLE proof}

Condition on the observation design $\cD$.
Since $1+\kappa\le2\kappa^2$, the condition \eqref{eq:general-baseline} implies the connectivity condition \eqref{eq:connectivity-baseline}.
By \Cref{lem:general-directed-connectivity,lem:mle-existence}, the MLE exists uniquely.
\Cref{lem:mle-score-localization} gives \eqref{eq:mle-localized-score}.
A union bound shows that both conclusions hold with probability at least $1-Cn^{-10}$.

Let $m$ be a median of the coordinates $\delta_1,\ldots,\delta_n$ of $\bm\delta$ and choose it so that
\[
  |\{i:\delta_i>m\}|\le n/2, \qquad |\{i:\delta_i<m\}|\le n/2.
\]
\runinhead{Upper tail}
For $t\ge0$, let $s_+(t)=|\{i:\delta_i>m+t\}|$.
For $s_+(t)>0$, apply \eqref{eq:mle-l2-truncation} in \Cref{lem:mle-truncated-energy} to $v_i=\min\{(\delta_i-m-t)_+,a\}$, where $0<a\le1$ is to be determined.
Since $v_i=a$ on $\{\delta_i>m+t+a\}$,
\begin{equation}
  a^2s_+(t+a) \le \|\bm{v}\|_2^2
  \le C\frac{\kappa^2}{\lambda_{\cD}}
  s_+(t)\log n.
  \label{eq:mle-level-recurrence}
\end{equation}

Set
\[
  a=A\kappa\sqrt{\frac{\log n}{\lambda_{\cD}}}, \qquad t_0=0.
\]
Here $A$ is a sufficiently large universal constant.
While $s_+(t_r)>0$, define $t_{r+1}=t_r+a$.
The condition \eqref{eq:general-baseline}, with a sufficiently large constant, ensures $a\le1$.
Substituting the definition of $a$ into \eqref{eq:mle-level-recurrence} and choosing $A$ large enough gives
\[
  s_+(t_{r+1})\le\frac14s_+(t_r).
\]
Define the actual stopping time
\[
  R_+:=\inf\{r\ge0:s_+(t_r)=0\}.
\]
Then $R_+\le\lceil\log_4 n\rceil+1$ and
\begin{equation}
  \max_i\delta_i-m
  \le t_{R_+}
  \le C\kappa\frac{\log^{3/2} n}{\sqrt{\lambda_{\cD}}}.
  \label{eq:mle-upper-tail}
\end{equation}

\runinhead{Lower tail}
Apply the lower-truncation version of \Cref{lem:mle-truncated-energy} to $s_-(t)=|\{i:\delta_i<m-t\}|$.
The lower-truncation analogue of \eqref{eq:mle-level-recurrence}, with the same step size $a$ from the upper iteration and its own stopping time $R_-$, gives
\begin{equation}
  m-\min_i\delta_i
  \le C\kappa\frac{\log^{3/2} n}{\sqrt{\lambda_{\cD}}}.
  \label{eq:mle-lower-tail}
\end{equation}
\runinhead{Conclusion}
Let $B$ be the larger of the right-hand sides in \eqref{eq:mle-upper-tail}--\eqref{eq:mle-lower-tail}.
Every coordinate of $\bm\delta$ lies in $[m-B,m+B]$.
Since $\one_n^\top\bm\delta=0$, this interval contains zero and $|m|\le B$.
Thus $\norm{\bm\delta}_\infty\le2B$.
This proves the MLE bound in \eqref{eq:general-entrywise} and completes the MLE part of \Cref{thm:general-entrywise}.

%% file: sections/c_rank_centrality_proof.tex
\section{Proof of Theorem~\ref{thm:general-entrywise}: Rank Centrality}
\label{sec:rc-proof}

This appendix proves the Rank Centrality part of \Cref{thm:general-entrywise}.
We first establish irreducibility and a stationary-flow identity, then obtain truncation-energy bounds over the level sets.
Finally, we use a multiplicative level-set contraction argument to obtain entrywise control.

We start by introducing the necessary multiplicative notation.
Fix a connected observation design $\cD$ and write $\lambda_{\cD}=\lambda_2(\bm L_{\cD})$.
Because BTL probabilities are invariant to common rescaling of the positive weights, define
\begin{equation}
  \bar\alpha_i^\star
  :=\frac{\alpha_i^\star}{\min_j\alpha_j^\star}
  =\exp\!\left(\theta_i^\star-\min_j\theta_j^\star\right),
  \qquad 1\le \bar\alpha_i^\star\le\kappa.
  \label{eq:rc-rescaled-weights}
\end{equation}
Define the corresponding population probability that item $j$ defeats item $i$ by
\begin{equation}
  q_{ij}^\star
  = \frac{\alpha_j^\star}
  {\alpha_i^\star+\alpha_j^\star}
  =\frac{\bar\alpha_j^\star}
  {\bar\alpha_i^\star+\bar\alpha_j^\star}.
  \label{eq:rc-population-probability}
\end{equation}

\subsection{Irreducibility and stationary-flow formulation}

Recall that $D_{\mathrm{win}}$ is the directed win graph of the empirical comparisons, with $i\to j$ whenever item $j$ defeats item $i$ at least once.
By \eqref{eq:rank-centrality-matrix}, it is also the transition digraph of the empirical Rank Centrality chain.
Consequently, whenever $D_{\mathrm{win}}$ is strongly connected, the chain is irreducible and \eqref{eq:rc-estimator} defines a unique Rank Centrality estimator.

Recall that in Section~\ref{sec:model} we chose an arbitrary orientation for each edge in the comparison graph $G$.
For each chosen orientation $e=(i,j)$, let $\bar e=(j,i)$ denote its reverse orientation and set
\[
  \bm b_{\bar e}:= \bm e_j-\bm e_i = -\bm b_e.
\]
Define the directed empirical edge flows at the true weights and their imbalance by
\begin{equation}
  f_e:=L_e\cdot\bar\alpha_i^\star\cdot\wh q_{ij},
  \qquad
  f_{\bar e}:=L_e\cdot\bar\alpha_j^\star\cdot\wh q_{ji},
  \qquad
  \gamma_e:=f_e-f_{\bar e}.
  \label{eq:rc-edge-flows}
\end{equation}
We similarly denote $\gamma_{\bar e}:=f_{\bar e}-f_e=-\gamma_e$.
Since $\wh q_{ji}=1-\wh q_{ij}$ and $1\le\bar\alpha_i^\star\le\kappa$, their total mass satisfies
\begin{equation}
  L_e\le f_e+f_{\bar e} \le \kappa L_e \label{eq:rc-edge-mass}
\end{equation}
for every edge $e\in E$.
Moreover, \eqref{eq:rc-population-probability} and $\E[\wh q_{ij}\mid\cD]=q_{ij}^\star$ give the centered imbalance identity
\begin{equation}
  \gamma_e
  =L_e\cdot(\bar\alpha_i^\star+\bar\alpha_j^\star)
  \cdot(\wh q_{ij}-q_{ij}^\star),
  \qquad
  \E[\gamma_e\mid\cD]=0.
  \label{eq:rc-centered-edge-noise}
\end{equation}

The variables $f_e$ and $f_{\bar e}$ isolate the flows in the two directions with deterministic bounds on their sum, while $\gamma_e$ captures their centered random imbalance.
The following lemma establishes concentration for the noise variables $\gamma_e$.
\begin{lemma}[Rank Centrality edge-noise localization]
  \label{lem:rc-noise-localization}
  Define
  \begin{equation}
    \bm g_{\RC}:=\sum_{e\in E}\gamma_e\cdot\bm b_e.
    \label{eq:rc-score}
  \end{equation}
  Conditional on the observation design $\cD$, with probability at least $1-Cn^{-10}$,
  \begin{equation}
    \abs{\bm g_{\RC}^\top\bm v}
    \le C\kappa\sqrt{s\log n}
    \sqrt{\bm v^\top\bm L_{\cD}\bm v}
    \label{eq:rc-localized-score}
  \end{equation}
  simultaneously for all $\bm v$ supported on a set of size $s\le n/2$.
\end{lemma}

\begin{proof}
  The edgewise noise variables $\gamma_e$ are independent.
  For every edge $e=(i,j)$ and repetition $\ell\in[L_e]$, write
  \[
    X_e^{(\ell)}=\ind\{j\text{ defeats }i\text{ in repetition }\ell\}, \qquad \wh q_{ij}=L_e^{-1}\sum_{\ell=1}^{L_e}X_e^{(\ell)}.
  \]
  Hoeffding's lemma gives, for every $t\in\R$,
  \begin{equation}
    \E\!\left[\exp(t\gamma_e)\mid\cD\right]
    \le \exp\!\left(
    \frac{t^2L_e\cdot(\bar\alpha_i^\star+\bar\alpha_j^\star)^2}{8}
    \right)
    \le \exp\!\left(\frac{t^2\kappa^2L_e}{2}\right).
    \label{eq:rc-noise-mgf}
  \end{equation}
  Applying \Cref{lem:localized-process} with $\tau=\kappa$ proves \eqref{eq:rc-localized-score}.
\end{proof}

Now suppose the empirical chain is irreducible, so $\wh\pi_i>0$ for every $i$.
Define the relative stationary masses by
\begin{equation}
  r_i:=\xi\cdot\frac{\wh\pi_i}{\bar\alpha_i^\star}>0,
  \label{eq:rc-relative-stationary}
\end{equation}
where the normalization factor $\xi>0$ is arbitrary.
By \eqref{eq:rank-centrality-matrix}, the $i$th coordinate of the stationary equation $\wh{\bm\pi}^\top\wh{\bm P}=\wh{\bm\pi}^\top$ is equivalent to
\[
  \sum_{j:\{i,j\}\in E}L_{ij} \left(\wh\pi_i\cdot\wh q_{ij}-\wh\pi_j\cdot\wh q_{ji}\right)=0.
\]
Multiplying this by $\xi$ and substituting \eqref{eq:rc-relative-stationary} gives the stationary-flow identity
\begin{equation}
  \sum_{j:\{i,j\}\in E}
  L_{ij} \left(\bar\alpha_i^\star\cdot r_i\cdot\wh q_{ij} -\bar\alpha_j^\star\cdot r_j\cdot\wh q_{ji}\right)=0, \qquad i\in[n].
  \label{eq:rc-stationary-flow}
\end{equation}
The ratio $r_i$ measures empirical stationary mass relative to its population target $\bar\alpha_i^\star$, up to a normalization factor.
We will choose the normalization factor $\xi$ only to rescale $\bm r$ so that its median is equal to one.

\subsection{Upper cut integration}

The goal of this subsection is to control the upper level sets of the relative stationary masses $\bm r$.
Fix $t>0$ and $a>0$, and define the upper-truncation vector $\bm z\in\R^n$ coordinatewise by
\begin{equation}
  z_i=T_{t,a}^+(r_i):=\min\{(r_i-t)_+,a\}.
  \label{eq:rc-upper-truncation}
\end{equation}
The vector $\bm z$ records the excess above $t$, truncated at height $a$, and equals $a$ on $\{i:r_i>t+a\}$.
Thus an upper bound on its Laplacian energy, combined with the sparse Poincar\'e inequality in \Cref{lem:sparse-poincare}, will control the size of the next upper level set.
We obtain this energy bound by integrating the stationary-flow identity over the intermediate level sets $S_h^+=\{i:r_i>h\}$ for $h\in[t,t+a]$.

\begin{lemma}[Upper stationary-flow truncation energy]
  \label{lem:rc-upper-truncation-energy}
  Assume that \eqref{eq:rc-localized-score} in \Cref{lem:rc-noise-localization} holds.
  Suppose $1\le s=|\{i:r_i>t\}|\le n/2$.
  Then the truncation in \eqref{eq:rc-upper-truncation} satisfies
  \begin{equation}
    \bm z^\top\bm L_{\cD}\bm z
    \le C\kappa^2\cdot(t+a)^2s\log n.
    \label{eq:rc-upper-truncation-energy}
  \end{equation}
\end{lemma}

\begin{proof}
Recall that $\partial S$ denotes the edge boundary of $S$.
Let $\one_S\in\{0,1\}^n$ denote the indicator vector of $S$, and define
\begin{equation}
  \Gamma(S):=\one_S^\top\bm g_{\RC}.
  \label{eq:rc-cut-noise}
\end{equation}
Recall that $\{r_i\}$ is defined in \eqref{eq:rc-relative-stationary}.
For $h>0$, let $S_h^+=\{i:r_i>h\}$.
Orient every edge in $\partial S_h^+$ as $e=(i,j)$, where $i\in S_h^+$ and $j\notin S_h^+$, using the reverse direction of the fixed orientation when necessary.
Reversal changes the signs of both $\gamma_e$ and $\bm b_e$, so $\gamma_e\cdot\bm b_e$ is unchanged.
We may therefore use these outward orientations in \eqref{eq:rc-score}.
Since $\one_{S_h^+}^\top\bm b_e=1$ on every outward boundary edge and is zero on every nonboundary edge,
\[
  \Gamma(S_h^+)=\sum_{e\in\partial S_h^+}\gamma_e.
\]

Define
\[
  P^+(h):=\sum_{e\in\partial S_h^+} \left\{f_e\cdot(r_i-h)+f_{\bar e}\cdot(h-r_j)\right\}.
\]
Every summand is nonnegative because $r_i>h\ge r_j$ and the flows are nonnegative.
On the other hand, summing the stationary-flow identity \eqref{eq:rc-stationary-flow} over $S_h^+$ cancels every internal edge and leaves
\[
  0=\sum_{e\in\partial S_h^+} (f_e\cdot r_i-f_{\bar e}\cdot r_j) =P^+(h)+h\cdot\Gamma(S_h^+),
\]
where the second equality uses
\[
  f_e\cdot r_i-f_{\bar e}\cdot r_j =f_e\cdot(r_i-h)+f_{\bar e}\cdot(h-r_j) +h\cdot(f_e-f_{\bar e}).
\]
Consequently,
\begin{equation}
  P^+(h)=-h\cdot\Gamma(S_h^+)\ge0.
  \label{eq:rc-upper-cut}
\end{equation}

We next lower bound the integral of $P^+(h)$ by the Laplacian energy of $\bm z$.
Fix an edge $\{i,j\}\in E$.
If $r_i=r_j$, the edge never belongs to $\partial S_h^+$ and contributes zero to $\int_t^{t+a}P^+(h)\,dh$.
Otherwise, relabel its endpoints so that $x:=r_i>r_j=:y$.
The edge belongs to $\partial S_h^+$ exactly when $y\le h<x$.
Hence the relevant levels form the interval $I=[t,t+a]\cap[y,x)$, whose length is
\[
  \ell=T_{t,a}^+(x)-T_{t,a}^+(y)=|z_i-z_j|.
\]
Over these levels the edge is oriented from $i$ to $j$, so its contribution to $\int_t^{t+a}P^+(h)\,dh$ is
\[
  f_e\int_I(x-h)\,dh+f_{\bar e}\int_I(h-y)\,dh.
\]
If $\ell=0$, both integrals vanish.
Otherwise, write $I=[\alpha,\alpha+\ell]$ up to the choice of endpoints.
Since $y\le \alpha<\alpha+\ell\le x$,
\[
  \int_I(x-h)\,dh
  =\ell(x-\alpha)-\frac{\ell^2}{2}
  \ge\frac{\ell^2}{2},
  \qquad
  \int_I(h-y)\,dh
  =\ell(\alpha-y)+\frac{\ell^2}{2}
  \ge\frac{\ell^2}{2}.
\]
The edge-mass bound \eqref{eq:rc-edge-mass} therefore makes this contribution at least
\[
  \frac12(f_e+f_{\bar e})\ell^2
  \ge\frac12L_e\cdot(z_i-z_j)^2.
\]
Summing over the edges gives
\begin{equation}
  \int_t^{t+a}
  P^+(h)\,dh \ge\frac12\sum_{\{i,j\}\in E}L_{ij}\cdot(z_i-z_j)^2 =\frac12\bm z^\top\bm L_{\cD}\bm z.
  \label{eq:rc-upper-energy-lower}
\end{equation}

Define the scalar cut-integration transform
\begin{equation}
  \psi_{t,a}^+(x):=\int_t^{t+a}h\cdot\ind\{x>h\}\,dh
  =t\cdot T_{t,a}^+(x)+\frac12T_{t,a}^+(x)^2
  \label{eq:rc-upper-cut-transform}
\end{equation}
and write $\bm\psi_{t,a}^+(\bm r):=(\psi_{t,a}^+(r_i))_{i\in[n]}$ for its coordinatewise application to $\bm r$.
For each level $h$, the indicator of $S_h^+$ has coordinates $\ind\{r_i>h\}$.
Hence \eqref{eq:rc-cut-noise} and an interchange of the finite sum with the integral give
\begin{equation}
  \begin{aligned}
    \int_t^{t+a}h\cdot\Gamma(S_h^+)\,dh
    &=\sum_{i=1}^n g_{\RC,i}
      \int_t^{t+a}h\cdot\ind\{r_i>h\}\,dh \\
    &=\sum_{i=1}^n g_{\RC,i}\cdot\psi_{t,a}^+(r_i)
     =\bm g_{\RC}^\top\bm\psi_{t,a}^+(\bm r).
  \end{aligned}
  \label{eq:rc-upper-cut-integration}
\end{equation}
The vector $\bm\psi_{t,a}^+(\bm r)$ is supported on $\{r>t\}$.
Moreover, since $T_{t,a}^+(x),T_{t,a}^+(y)\in[0,a]$, the identity \eqref{eq:rc-upper-cut-transform} gives
\begin{equation}
  \begin{aligned}
    \abs{\psi_{t,a}^+(x)-\psi_{t,a}^+(y)}
    &=\abs{T_{t,a}^+(x)-T_{t,a}^+(y)}
      \left\{t+\frac{T_{t,a}^+(x)+T_{t,a}^+(y)}{2}\right\} \\
    &\le(t+a)\cdot\abs{T_{t,a}^+(x)-T_{t,a}^+(y)}.
  \end{aligned}
  \label{eq:rc-upper-lipschitz}
\end{equation}

  Combine \eqref{eq:rc-upper-cut}, \eqref{eq:rc-upper-energy-lower}, and \eqref{eq:rc-upper-cut-integration} to obtain
  \[
    \frac12\bm z^\top\bm L_{\cD}\bm z \le-\bm g_{\RC}^\top\bm\psi_{t,a}^+(\bm r).
  \]
  Since $\bm\psi_{t,a}^+(\bm r)$ is supported on $\{i:r_i>t\}$, whose size is $s$, we may apply \eqref{eq:rc-localized-score} from \Cref{lem:rc-noise-localization} with $\bm v=\bm\psi_{t,a}^+(\bm r)$ to obtain
  \[
    -\bm g_{\RC}^\top\bm\psi_{t,a}^+(\bm r)
    \le \abs{\bm g_{\RC}^\top\bm\psi_{t,a}^+(\bm r)}
    \le C\kappa\sqrt{s\log n}\cdot
    \sqrt{\bm\psi_{t,a}^+(\bm r)^\top\bm L_{\cD}
      \bm\psi_{t,a}^+(\bm r)}.
  \]
  Moreover, applying \eqref{eq:rc-upper-lipschitz} to every edge and using $z_i=T_{t,a}^+(r_i)$ gives
  \[
    \begin{aligned}
      \bm\psi_{t,a}^+(\bm r)^\top\bm L_{\cD}
        \bm\psi_{t,a}^+(\bm r)
      &=\sum_{\{i,j\}\in E}L_{ij}\cdot
        \abs{\psi_{t,a}^+(r_i)-\psi_{t,a}^+(r_j)}^2 \\
      &\le (t+a)^2\sum_{\{i,j\}\in E}L_{ij}\cdot(z_i-z_j)^2
       =(t+a)^2\bm z^\top\bm L_{\cD}\bm z.
    \end{aligned}
  \]
  Combining these two bounds with the preceding display yields
  \[
    \frac12\bm z^\top\bm L_{\cD}\bm z
    \le C\kappa\sqrt{s\log n}
    (t+a)\cdot\sqrt{\bm z^\top\bm L_{\cD}\bm z}.
  \]
  Rearranging this gives \eqref{eq:rc-upper-truncation-energy} and completes the proof.
\end{proof}

\subsection{Lower cut truncation}

For $t>0$ and $0<a<t$, define the lower-truncation vector $\bm z\in\R^n$ coordinatewise by
\begin{equation}
  z_i=T_{t,a}^-(r_i):=\min\{(t-r_i)_+,a\}.
  \label{eq:rc-lower-truncation}
\end{equation}
The corresponding energy bound follows from the upper-cut argument after a sign reversal.

\begin{lemma}[Lower stationary-flow truncation energy]
  \label{lem:rc-lower-truncation-energy}
  Assume that \eqref{eq:rc-localized-score} in \Cref{lem:rc-noise-localization} holds.
  Suppose $1\le s=|\{i:r_i<t\}|\le n/2$.
  Then the lower truncation in \eqref{eq:rc-lower-truncation} satisfies
  \begin{equation}
    \bm z^\top\bm L_{\cD}\bm z
    \le C\kappa^2t^2s\log n.
    \label{eq:rc-lower-truncation-energy}
  \end{equation}
\end{lemma}

\begin{proof}
  Set $\widetilde{\bm r}:=-\bm r$.
  The stationary-flow identity \eqref{eq:rc-stationary-flow} is linear in $\bm r$, so it also holds with $\bm r$ replaced by $\widetilde{\bm r}$.
  Apply the upper-cut argument from the proof of \Cref{lem:rc-upper-truncation-energy} to $\widetilde{\bm r}$ over the level interval $[-t,-t+a]$.
  The upper level set at the initial level and the resulting truncation are
  \[
    \{i:\widetilde r_i>-t\}=\{i:r_i<t\},
    \qquad
    T_{-t,a}^+(\widetilde r_i)=T_{t,a}^-(r_i)=z_i.
  \]
  Thus the localized-score bound applies with support size $s$.

  It remains only to identify the Lipschitz factor in the sign-reversed argument.
  Define
  \[
    \widetilde\psi(x):=\int_{-t}^{-t+a}
      h\cdot\ind\{x>h\}\,dh.
  \]
  Since $T_{-t,a}^+(x),T_{-t,a}^+(y)\in[0,a]$, the same factorization as in \eqref{eq:rc-upper-lipschitz} gives
  \[
    \begin{aligned}
      \abs{\widetilde\psi(x)-\widetilde\psi(y)}
      &=
      \abs{T_{-t,a}^+(x)-T_{-t,a}^+(y)}\cdot
      \abs{-t+\frac{T_{-t,a}^+(x)+T_{-t,a}^+(y)}{2}} \\
      &\le
      t\cdot\abs{T_{-t,a}^+(x)-T_{-t,a}^+(y)}.
    \end{aligned}
  \]
  The last inequality uses $a<t$.
  The remainder of the proof of \Cref{lem:rc-upper-truncation-energy}, with this Lipschitz factor, yields \eqref{eq:rc-lower-truncation-energy}.
\end{proof}

\subsection{Completion of the Rank Centrality proof}

Condition on the observation design $\cD$.
Since $1+\kappa\le2\kappa^2$, the baseline condition \eqref{eq:general-baseline} implies the connectivity condition \eqref{eq:connectivity-baseline} in \Cref{lem:general-directed-connectivity}.
Hence \Cref{lem:general-directed-connectivity} shows that the empirical Rank Centrality chain is irreducible, so the Rank Centrality estimator exists uniquely.
Moreover, \Cref{lem:rc-noise-localization} shows that \eqref{eq:rc-localized-score} holds.
A union bound shows that both conclusions hold with probability at least $1-Cn^{-10}$.

Irreducibility gives $\wh\pi_i>0$ for every $i$, so the normalization factor $\xi$ in \eqref{eq:rc-relative-stationary} can be chosen to make $1$ a median of $\bm r$:
\begin{equation}
  |\{i:r_i>1\}|\le n/2,
  \qquad
  |\{i:r_i<1\}|\le n/2.
  \label{eq:rc-median-scale}
\end{equation}

\runinhead{Upper tail}
For $t\ge1$, let $s_+(t)=|\{i:r_i>t\}|$.
Set
\begin{equation}
  \eta:=A\kappa\sqrt{\frac{\log n}{\lambda_{\cD}}},
  \label{eq:rc-step-size}
\end{equation}
where $A$ is a sufficiently large universal constant.
Under \eqref{eq:general-baseline}, with its universal constant enlarged, $\eta\le1/2$.
By \eqref{eq:rc-median-scale}, $s_+(t)\le n/2$.
If $s_+(t)>0$, take $a=\eta t$ in the upper truncation \eqref{eq:rc-upper-truncation}.
It equals $\eta t$ on $\{r>t\cdot(1+\eta)\}$, so \Cref{lem:sparse-poincare,lem:rc-upper-truncation-energy} give the level-set recurrence
\begin{equation}
  (\eta t)^2s_+\bigl(t\cdot(1+\eta)\bigr)
  \le \norm{\bm z}_2^2
  \le C\frac{\kappa^2}{\lambda_{\cD}}\,
      t^2(1+\eta)^2s_+(t)\log n.
  \label{eq:rc-upper-level-recurrence}
\end{equation}
Substituting \eqref{eq:rc-step-size} and choosing $A$ large enough yields
\begin{equation}
  s_+\bigl(t\cdot(1+\eta)\bigr)\le\frac14s_+(t).
  \label{eq:rc-upper-contraction}
\end{equation}

Starting from $t_0=1$, define $t_{k+1}=t_k\cdot(1+\eta)$ whenever $s_+(t_k)>0$.
Define the actual stopping time
\[
  R_+:=\inf\{k\ge0:s_+(t_k)=0\}.
\]
By \eqref{eq:rc-upper-contraction}, $R_+\le\lceil\log_4n\rceil+1$.
Since $s_+(t_{R_+})=0$ implies $\max_i r_i\le t_{R_+}$,
\begin{equation}
  \log\max_i r_i
  \le \log t_{R_+}
  =R_+\cdot\log(1+\eta)
  \le R_+\cdot\eta
  \le C\kappa\frac{\log^{3/2} n}{\sqrt{\lambda_{\cD}}}.
  \label{eq:rc-upper-final}
\end{equation}

\runinhead{Lower tail}
For $0<t\le1$, let $s_-(t)=|\{i:r_i<t\}|$.
By \eqref{eq:rc-median-scale}, $s_-(t)\le n/2$.
When this set is nonempty, take $a=\eta t$ in the lower truncation \eqref{eq:rc-lower-truncation}.
Since $\eta\le1/2$, the requirement $0<a<t$ is satisfied.
This truncation equals $\eta t$ on $\{r<t\cdot(1-\eta)\}$.
Thus \Cref{lem:sparse-poincare,lem:rc-lower-truncation-energy} give the lower-truncation analogue of \eqref{eq:rc-upper-level-recurrence}.
Together with \eqref{eq:rc-step-size}, this yields
\begin{equation}
  s_-\bigl(t\cdot(1-\eta)\bigr)\le\frac14s_-(t).
  \label{eq:rc-lower-contraction}
\end{equation}

Starting from $t_0=1$, define $t_{k+1}=t_k\cdot(1-\eta)$ whenever $s_-(t_k)>0$.
Let $R_-:=\inf\{k\ge0:s_-(t_k)=0\}$ be the corresponding stopping time.
By \eqref{eq:rc-lower-contraction}, $R_-\le\lceil\log_4n\rceil+1$.
Since $s_-(t_{R_-})=0$ implies $\min_i r_i\ge t_{R_-}$, and since $-\log(1-x)\le2x$ for $x\le1/2$,
\begin{align}
  -\log\min_i r_i
   & \le-\log t_{R_-}
  =R_-\cdot[-\log(1-\eta)] \\
   & \le2R_-\cdot\eta
  \le C\kappa\frac{\log^{3/2} n}{\sqrt{\lambda_{\cD}}}.
  \label{eq:rc-lower-final}
\end{align}
Consequently,
\begin{equation}
  \max_i\log r_i-\min_i\log r_i
  \le C\kappa\frac{\log^{3/2} n}{\sqrt{\lambda_{\cD}}}.
  \label{eq:rc-log-range}
\end{equation}

\runinhead{Conclusion}
Finally, \eqref{eq:rc-rescaled-weights} and \eqref{eq:rc-relative-stationary} give
\[
  \log r_i=\log\wh\pi_i-\theta_i^\star+\text{constant}.
\]
After centering,
\[
  \wh{\bm\theta}^{\RC}-\bm\theta^\star
  =\log\bm r-\frac1n(\one_n^\top\log\bm r)\one_n.
\]
The mean of $\log\bm r$ lies between its minimum and maximum, so the $\ell_\infty$ norm of this centered vector is at most $\max_i\log r_i-\min_i\log r_i$.
Equation \eqref{eq:rc-log-range} proves the Rank Centrality bound in \eqref{eq:general-entrywise}.
Combining this event with the MLE event from \Cref{sec:mle-proof} and applying a union bound proves \Cref{thm:general-entrywise}.

%% file: sections/d_outcome_adaptive_proofs.tex
\section{Proofs for outcome-adaptive augmentation}
\label{sec:completion}

Recall the initial design $\cD_0$ and final design $\cD$ from \Cref{sec:random-design-models}.
The proof uses the two stages separately: the initial outcomes concentrate with $\cD_0$ fixed, and the fresh outcomes concentrate conditional on the initial data and the chosen augmentation.
We use the resulting score-localization bounds together with Laplacian monotonicity to prove the general outcome-adaptive guarantee.

\begin{lemma}[Algebraic connectivity of an Erd\H{o}s--R\'enyi graph]
  \label{lem:er-gap}
  If $G_0\sim G(n,p)$ and $np\ge C\log n$, then, with probability at least $1-Cn^{-10}$,
  \begin{equation}
    \lambda_2(\bm L_{G_0})\ge\frac12np.
    \label{eq:er-gap}
  \end{equation}
  Consequently, for any integer $L\ge1$, every observation design $\cD$ with at least $L$ comparisons on each edge of $G_0$ satisfies $\lambda_{\cD}\ge npL/2$ on this event.
\end{lemma}

\begin{proof}
  The first claim follows from \citet[Lemma~6.4]{ChenFanMaWang2019}.
  If $\cD$ contains at least $L$ observations on every edge of $G_0$, then, for every $\bm v\perp\one_n$,
  \[
    \bm v^\top\bm L_{\cD}\bm v
    \ge L\bm v^\top\bm L_{G_0}\bm v.
  \]
  Taking the minimum over unit vectors orthogonal to $\one_n$ gives
  \[
    \lambda_{\cD}
    \ge L\lambda_2(\bm L_{G_0})
    \ge\frac12npL.
  \]
\end{proof}

\begin{lemma}[Two-stage localized score control]
  \label{lem:adaptive-split-localization}
  Under \Cref{def:semirandom}, suppose that $G_0$ is connected.
  Let $\bm g_{\MLE}$ and $\bm g_{\RC}$ be the final-design score vectors in \eqref{eq:mle-score} and \eqref{eq:rc-score}.
  Uniformly over the adversary, with probability at least $1-Cn^{-10}$, simultaneously for every vector $\bm v$ supported on a set of size $s\le n/2$,
  \begin{align}
    \abs{\bm g_{\MLE}^\top\bm v}
     &\le C\sqrt{s\log n}
       \sqrt{\bm v^\top\bm L_{\cD}\bm v},
    \label{eq:adaptive-mle-localization} \\
    \abs{\bm g_{\RC}^\top\bm v}
     &\le C\kappa\sqrt{s\log n}
       \sqrt{\bm v^\top\bm L_{\cD}\bm v}.
    \label{eq:adaptive-rc-localization}
  \end{align}
\end{lemma}

\begin{proof}
Because $G_0$ is connected and all initial edge counts are positive, $\cD_0$ is connected.
Fix an orientation for every pair.
For $a\in\{\MLE,\RC\}$, decompose the final score as
\[
  \bm g_a=\bm g_{a,0}+\bm g_{a,+},
  \qquad \tau_{\MLE}=1,\qquad \tau_{\RC}=\kappa,
\]
where the two terms contain the initial and fresh score contributions, respectively.
With $\cD_0$ fixed, the initial edge contributions are independent and centered.
Hoeffding's lemma and \Cref{lem:localized-process}, applied to $\cD_0$, give, except with probability $Cn^{-10}$, simultaneously for every $\bm v$ supported on a set of size $s\le n/2$,
\begin{equation}
  \abs{\bm g_{a,0}^\top\bm v}
  \le C\tau_a\sqrt{s\log n}
  \sqrt{\bm v^\top\bm L_{\cD_0}\bm v}.
  \label{eq:adaptive-initial-localization}
\end{equation}
This bound holds over all such test vectors, so it remains usable after the augmentation is chosen.

Let $\mathcal F_{\mathrm{aug}}$ be the sigma-field generated by $\cD_0$, the initial outcomes, and the chosen counts $\{M_{ij}\}$.
Conditional on $\mathcal F_{\mathrm{aug}}$, the fresh score contributions are independent and centered.
The same Hoeffding bounds as in the MLE and Rank Centrality proofs apply with counts $M_{ij}\le L_{ij}$.
Applying \Cref{lem:localized-process} conditionally to the connected final design $\cD$ therefore gives, except with conditional probability $Cn^{-10}$,
\begin{equation}
  \abs{\bm g_{a,+}^\top\bm v}
  \le C\tau_a\sqrt{s\log n}
  \sqrt{\bm v^\top\bm L_{\cD}\bm v}
  \label{eq:adaptive-fresh-localization}
\end{equation}
simultaneously over the same class of test vectors.
The augmentation graph itself need not be connected.

Since $\bm L_{\cD_0}\preceq\bm L_{\cD}$, adding the two bounds proves \eqref{eq:adaptive-mle-localization} and \eqref{eq:adaptive-rc-localization}.
Taking a union bound over the two stages and two estimators gives probability at least $1-Cn^{-10}$.
\end{proof}

\subsection{Proof of Theorem~\ref{thm:semirandom}}
\label{sec:semirandom-theorem-proof}
By \eqref{eq:semirandom-baseline}, $\lambda_{\cD_0}\ge C\kappa^2\log n$.
Since $M_{ij}\ge0$,
\[
  \bm L_{\cD}-\bm L_{\cD_0}
  =\sum_{1\le i<j\le n}M_{ij}
    (\bm e_i-\bm e_j)(\bm e_i-\bm e_j)^\top
  \succeq\bm 0,
  \qquad
  \lambda_{\cD}\ge\lambda_{\cD_0}.
\]

\noindent\textbf{Well-posedness.}
Because $1+\kappa\le2\kappa^2$, the initial spectral condition implies \eqref{eq:connectivity-baseline} for $\cD_0$.
Applying \Cref{lem:general-directed-connectivity} to the initial comparisons therefore shows that their directed win graph is strongly connected except with probability $Cn^{-10}$.
All its arcs remain in the final directed win graph.
Hence the MLE exists uniquely by \Cref{lem:mle-existence}.
The empirical Rank Centrality chain is irreducible, so the Rank Centrality estimator also exists uniquely.

\noindent\textbf{Entrywise control.}
By \Cref{lem:adaptive-split-localization}, the final MLE and Rank Centrality scores satisfy their respective localized-score bounds uniformly over the adversarial augmentation.
Repeating the deterministic level-set arguments in \Cref{sec:mle-proof,sec:rc-proof}, now with these score bounds and the final design $\cD$, gives
\[
  \max_{a\in\{\MLE,\RC\}}
  \norm{\wh{\bm\theta}^{a}-\bm\theta^\star}_\infty
  \le C_{\mathrm{inf}}\kappa
  \frac{\log^{3/2}n}{\sqrt{\lambda_{\cD}}}.
\]
A union bound over all high-probability events gives the stated probability.

\subsection{A coupling for heterogeneous sampling}
\label{sec:heterogeneous-coupling}
\begin{proposition}[Heterogeneous sampling is semirandom]
  \label{prop:heterogeneous-is-semirandom}
  Fix $p\in(0,1]$ and an integer $L\ge1$.
  Suppose each pair $\{i,j\}$ is observed independently with probability $p_{ij}\in[p,1]$.
  On each observed pair, generate $L$ independent BTL comparisons.
  Then the joint law of the observed design and comparison outcomes can be realized under \Cref{def:semirandom} after drawing $G_0\sim G(n,p)$ and assigning $L$ initial comparisons to each edge of $G_0$.
\end{proposition}

\begin{proof}
  Independently for each pair $\{i,j\}$, draw $U_{ij}\sim\operatorname{Unif}[0,1]$ and set
  \[
    B_{ij}=\ind\{U_{ij}\le p\},
    \qquad
    A_{ij}=\ind\{U_{ij}\le p_{ij}\}.
  \]
  Since $p_{ij}\ge p$, the variables $B_{ij}$ form an initial $G(n,p)$ graph contained in the heterogeneous graph defined by the $A_{ij}$.
  Generate $L$ initial comparisons on its edges and set
  \[
    M_{ij}=L(A_{ij}-B_{ij}).
  \]
  Since $B_{ij}\le A_{ij}$, we have $M_{ij}\ge0$, and the final count on each pair is $LA_{ij}$.
  Generating the $M_{ij}$ additional comparisons in the augmentation stage gives the heterogeneous-sampling law.
\end{proof}

%% file: sections/e_experimental_implementation.tex
\section{Experimental implementation}
\label{sec:experimental-details}

This appendix supplements the experimental descriptions in \Cref{sec:experiments}.
See the code for the exact implementation. 

\paragraph*{Degree-matched Arena replay.}

The original data contain four outcome categories: a win for model A, a win for model B, a tie, or a judgment that both models are bad.
We only keep the decisive votes for the BTL analysis.
In each trial, we use a fixed set of held-out reference votes. For both local repetition and global linking, we start with the same initial graph and grow it at each step. Each larger design inherits the observations from the smaller design.

\paragraph*{Complete-multipartite simulation.}

In each trial, independent Gaussian scores are centered and rescaled to have range one. All designs use the same latent score vector.

\paragraph*{Monotone-augmentation simulation.}

For the entrywise-error panel, the latent scores are equally spaced between $-1/2$ and $1/2$.
For exact recovery, we assign score $\Delta_K/2$ to $K-1$ items and $-\Delta_K/2$ to $n-K-1$ items, with the remaining two items assigned scores $0.75$ and $-0.75$.
We then subtract the mean from all scores, preserving the top-$K$ gap $\Delta_K$ and the total score range $1.5$.
Item labels are randomly permuted in every trial.

The SDP weights maximize weighted algebraic connectivity subject to $0\le w_{ij}\le1$ and $\sum_j w_{ij}\le2np_0$ for each item, where $p_0$ is the base-edge probability.
We solve the SDP using deterministic cutting-plane linear programs.

%% file: references.bib
@article{Zermelo1929,
  author  = {Zermelo, Ernst},
  title   = {{Die Berechnung der Turnier-Ergebnisse als ein Maximumproblem der Wahrscheinlichkeitsrechnung}},
  journal = {Mathematische Zeitschrift},
  volume  = {29},
  number  = {1},
  pages   = {436--460},
  year    = {1929},
  doi     = {10.1007/BF01180541}
}

@article{BradleyTerry1952,
  author  = {Bradley, Ralph Allan and Terry, Milton E.},
  title   = {Rank Analysis of Incomplete Block Designs: {I}. {The} Method of Paired Comparisons},
  journal = {Biometrika},
  volume  = {39},
  number  = {3--4},
  pages   = {324--345},
  year    = {1952},
  doi     = {10.1093/biomet/39.3-4.324}
}

@book{Luce1959,
  author    = {Luce, R. Duncan},
  title     = {Individual Choice Behavior: A Theoretical Analysis},
  publisher = {Wiley},
  address   = {New York},
  year      = {1959}
}

@article{Ford1957,
  author  = {Ford, Jr., Lester R.},
  title   = {Solution of a Ranking Problem from Binary Comparisons},
  journal = {The American Mathematical Monthly},
  volume  = {64},
  number  = {8, Part 2},
  pages   = {28--33},
  year    = {1957},
  doi     = {10.2307/2308513}
}

@article{Cattelan2012,
  author  = {Cattelan, Manuela},
  title   = {Models for Paired Comparison Data: A Review with Emphasis on Dependent Data},
  journal = {Statistical Science},
  volume  = {27},
  number  = {3},
  pages   = {412--433},
  year    = {2012},
  doi     = {10.1214/12-STS396}
}

@article{NegahbanOhShah2017,
  author  = {Negahban, Sahand and Oh, Sewoong and Shah, Devavrat},
  title   = {{Rank Centrality}: Ranking from Pairwise Comparisons},
  journal = {Operations Research},
  volume  = {65},
  number  = {1},
  pages   = {266--287},
  year    = {2017},
  doi     = {10.1287/opre.2016.1534}
}

@inproceedings{ChenSuh2015,
  author    = {Chen, Yuxin and Suh, Changho},
  title     = {Spectral {MLE}: Top-{$K$} Rank Aggregation from Pairwise Comparisons},
  booktitle = {Proceedings of the 32nd International Conference on Machine Learning},
  series    = {Proceedings of Machine Learning Research},
  volume    = {37},
  pages     = {371--380},
  year      = {2015},
  url       = {https://proceedings.mlr.press/v37/chena15.html}
}

@misc{JangKimSuhOh2016,
  author       = {Jang, Minje and Kim, Sunghyun and Suh, Changho and Oh, Sewoong},
  title        = {Top-{$K$} Ranking from Pairwise Comparisons: When Spectral Ranking Is Optimal},
  howpublished = {arXiv preprint arXiv:1603.04153},
  year         = {2016},
  url          = {https://arxiv.org/abs/1603.04153}
}

@article{ChenFanMaWang2019,
  author  = {Chen, Yuxin and Fan, Jianqing and Ma, Cong and Wang, Kaizheng},
  title   = {Spectral Method and Regularized {MLE} Are Both Optimal for top-{$K$} Ranking},
  journal = {The Annals of Statistics},
  volume  = {47},
  number  = {4},
  pages   = {2204--2235},
  year    = {2019},
  doi     = {10.1214/18-AOS1745}
}

@article{ChenGaoZhang2022,
  author  = {Chen, Pinhan and Gao, Chao and Zhang, Anderson Y.},
  title   = {Partial Recovery for top-{$k$} Ranking: Optimality of {MLE} and Suboptimality of the Spectral Method},
  journal = {The Annals of Statistics},
  volume  = {50},
  number  = {3},
  pages   = {1618--1652},
  year    = {2022},
  doi     = {10.1214/21-AOS2166}
}

@inproceedings{LiShrotriyaRinaldo2022,
  author    = {Li, Wanshan and Shrotriya, Shamindra and Rinaldo, Alessandro},
  title     = {$\ell_\infty$-Bounds of the {MLE} in the {BTL} Model under General Comparison Graphs},
  booktitle = {Proceedings of the Thirty-Eighth Conference on Uncertainty in Artificial Intelligence},
  series    = {Proceedings of Machine Learning Research},
  volume    = {180},
  pages     = {1178--1187},
  year      = {2022},
  url       = {https://proceedings.mlr.press/v180/li22g.html}
}

@misc{LiXu2026,
  author       = {Li, Yicheng and Xu, Huifu},
  title        = {Error Bounds for Statistical Estimators in {BTL} Model with Parametric Multivariate Utility Functions},
  howpublished = {arXiv preprint arXiv:2609.26326},
  year         = {2026},
  url          = {https://arxiv.org/abs/2609.26326}
}

@article{YanYangXu2012,
  author  = {Yan, Ting and Yang, Yaning and Xu, Jinfeng},
  title   = {Sparse Paired Comparisons in the {Bradley--Terry} Model},
  journal = {Statistica Sinica},
  volume  = {22},
  number  = {3},
  pages   = {1305--1318},
  year    = {2012},
  doi     = {10.5705/ss.2010.299}
}

@misc{Chen2023General,
  author       = {Chen, Yanxi},
  title        = {Ranking from Pairwise Comparisons in General Graphs and Graphs with Locality},
  howpublished = {arXiv preprint arXiv:2304.06821},
  year         = {2023},
  url          = {https://arxiv.org/abs/2304.06821}
}

@article{HanYeTanChen2020,
  author  = {Han, Ruijian and Ye, Rougang and Tan, Chunxi and Chen, Kani},
  title   = {Asymptotic Theory of Sparse {Bradley--Terry} Model},
  journal = {The Annals of Applied Probability},
  volume  = {30},
  number  = {5},
  pages   = {2491--2515},
  year    = {2020},
  doi     = {10.1214/20-AAP1564}
}

@article{HanXuChen2023,
  author  = {Han, Ruijian and Xu, Yiming and Chen, Kani},
  title   = {A General Pairwise Comparison Model for Extremely Sparse Networks},
  journal = {Journal of the American Statistical Association},
  volume  = {118},
  number  = {544},
  pages   = {2422--2432},
  year    = {2023},
  doi     = {10.1080/01621459.2022.2053137}
}

@article{GaoShenZhang2023,
  author  = {Gao, Chao and Shen, Yandi and Zhang, Anderson Y.},
  title   = {Uncertainty Quantification in the {Bradley--Terry--Luce} Model},
  journal = {Information and Inference: A Journal of the IMA},
  volume  = {12},
  number  = {2},
  pages   = {1073--1140},
  year    = {2023},
  doi     = {10.1093/imaiai/iaac032}
}

@inproceedings{YangChenOrecchiaMa2024,
  author    = {Yang, Yuepeng and Chen, Antares and Orecchia, Lorenzo and Ma, Cong},
  title     = {Top-{$K$} Ranking with a Monotone Adversary},
  booktitle = {Proceedings of Thirty Seventh Conference on Learning Theory},
  series    = {Proceedings of Machine Learning Research},
  volume    = {247},
  pages     = {5123--5162},
  year      = {2024},
  url       = {https://proceedings.mlr.press/v247/yang24b.html}
}

@inproceedings{LeeMakurSingh2026,
  author    = {Lee, Dongmin and Makur, Anuran and Singh, Japneet},
  title     = {Entrywise Error Bounds for Spectral Ranking with Semi-Random Adversaries},
  booktitle = {Proceedings of the 32nd ACM SIGKDD Conference on Knowledge Discovery and Data Mining V.2},
  pages     = {2449--2460},
  year      = {2026},
  doi       = {10.1145/3770855.3818135}
}

@article{YangMa2026Rasch,
  author  = {Yang, Yuepeng and Ma, Cong},
  title   = {Random Pairing {MLE} for Estimation of Item Parameters in {Rasch} Model},
  journal = {Journal of the American Statistical Association},
  volume  = {121},
  number  = {554},
  pages   = {1683--1694},
  year    = {2026},
  doi     = {10.1080/01621459.2025.2582873}
}

@inproceedings{NguyenZhang2023,
  author    = {Nguyen, Duc and Zhang, Anderson Ye},
  title     = {Optimal and Private Learning from Human Response Data},
  booktitle = {Proceedings of the 26th International Conference on Artificial Intelligence and Statistics},
  series    = {Proceedings of Machine Learning Research},
  volume    = {206},
  pages     = {922--958},
  year      = {2023},
  url       = {https://proceedings.mlr.press/v206/nguyen23a.html}
}

@article{FanLouWangYu2026,
  author  = {Fan, Jianqing and Lou, Zhipeng and Wang, Weichen and Yu, Mengxin},
  title   = {Spectral Ranking Inferences Based on General Multiway Comparisons},
  journal = {Operations Research},
  volume  = {74},
  number  = {1},
  pages   = {161--180},
  year    = {2026},
  doi     = {10.1287/opre.2023.0439}
}

@inproceedings{MohajerSuhElmahdy2017,
  author    = {Mohajer, Soheil and Suh, Changho and Elmahdy, Adel},
  title     = {Active Learning for top-{$K$} Rank Aggregation from Noisy Comparisons},
  booktitle = {Proceedings of the 34th International Conference on Machine Learning},
  series    = {Proceedings of Machine Learning Research},
  volume    = {70},
  pages     = {2488--2497},
  year      = {2017},
  url       = {https://proceedings.mlr.press/v70/mohajer17a.html}
}

@inproceedings{GoldbergerRudi2026,
  author    = {Goldberger, Motti and Rudi, Nils},
  title     = {Optimal Top-{$k$} Identification from Pairwise Comparisons},
  booktitle = {Proceedings of the 43rd International Conference on Machine Learning},
  series    = {Proceedings of Machine Learning Research},
  volume    = {306},
  pages     = {35572--35599},
  year      = {2026},
  url       = {https://proceedings.mlr.press/v306/goldberger26a.html}
}

@misc{JinYu2026,
  author       = {Jin, Yiqiao and Yu, Mengxin},
  title        = {Maximum Likelihood Estimation for Entity Ranking under Iterative Synthetic Data Augmentation},
  howpublished = {arXiv preprint arXiv:2609.32987},
  year         = {2026},
  url          = {https://arxiv.org/abs/2609.32987}
}

@article{Grossmann2019PartialProfile,
  author  = {Gro{\ss}mann, Heiko},
  title   = {A Practical Approach to Designing Partial-Profile Choice
             Experiments with Two Alternatives for Estimating Main Effects
             and Interactions of Many Two-Level Attributes},
  journal = {Journal of Choice Modelling},
  volume  = {32},
  pages   = {100136},
  year    = {2019},
  doi     = {10.1016/j.jocm.2018.04.001}
}

@inproceedings{ChiangEtAl2024Arena,
  author    = {Chiang, Wei-Lin and Zheng, Lianmin and Sheng, Ying and
               Angelopoulos, Anastasios Nikolas and Li, Tianle and Li, Dacheng
               and Zhu, Banghua and Zhang, Hao and Jordan, Michael I. and
               Gonzalez, Joseph E. and Stoica, Ion},
  title     = {{Chatbot Arena}: An Open Platform for Evaluating {LLM}s by Human
               Preference},
  booktitle = {Proceedings of the 41st International Conference on Machine
               Learning},
  series    = {Proceedings of Machine Learning Research},
  volume    = {235},
  pages     = {8359--8388},
  year      = {2024},
  url       = {https://proceedings.mlr.press/v235/chiang24b.html}
}

@misc{LMArenaHumanPreference140k,
  author       = {{LMArena}},
  title        = {{Arena Human Preference 140K}},
  year         = {2025},
  howpublished = {Hugging Face data set},
  url          = {https://huggingface.co/datasets/lmarena-ai/arena-human-preference-140k},
  note         = {Dataset revision recorded in the artifact metadata; accessed August 8, 2026}
}
